\documentclass[11pt,a4paper]{preprint}

\usepackage[margin = 3cm]{geometry}
\usepackage{soul} 

\usepackage[full]{textcomp}

\usepackage[utf8]{inputenc}
\usepackage[T1]{fontenc}

\usepackage{newtxtext} 

\usepackage{amstext, amsbsy, amssymb, amsmath, amsthm, amsfonts}

\usepackage[toc,page]{appendix}

\usepackage[english]{babel}

\usepackage{bbm}

\usepackage{bm}

\usepackage{booktabs}

\usepackage{cancel} 

\usepackage{caption}

\usepackage{commath}

\usepackage{enumitem}

\usepackage{graphicx}

\usepackage{hyperref} 
\hypersetup{colorlinks=true, urlcolor= black, linkcolor=blue, citecolor=teal}

\usepackage{longtable}

\usepackage{marginnote}

\usepackage{mathrsfs}

\usepackage{mathtools}

\usepackage{nicefrac}

\usepackage{subcaption}

\usepackage{tcolorbox}
\usepackage{import}
\usepackage{xifthen}
\usepackage{pdfpages}
\usepackage{transparent}

\theoremstyle{plain}
\newtheorem{thm}{Theorem}[section]
\newtheorem{lem}[thm]{Lemma}

\newtheorem{prop}[thm]{Proposition}
\newtheorem{cor}[thm]{Corollary}

\theoremstyle{definition}
\newtheorem{definition}[thm]{Definition}

\theoremstyle{remark}

\newcommand{\be}[1]{\begin{equation}\label{#1}}
	\newcommand{\ee}{\end{equation}}
\numberwithin{equation}{section}

\newcommand{\ba}[1]{\begin{align}\label{#1}}
	\newcommand{\ea}{\end{align}}
\numberwithin{equation}{section}

\newcommand{\ben}{\begin{equation*}}
	\newcommand{\een}{\end{equation*}}
\numberwithin{equation}{section}

\mathtoolsset{showonlyrefs, showmanualtags}

\newcommand{\scrC}{\mathscr{C}}

\newcommand{\calE}{\mathcal{E}}
\newcommand{\calF}{\mathcal{F}}

\newcommand{\calH}{\mathcal{H}}

\newcommand{\calW}{\mathcal{W}}

\newcommand{\bbE}{\mathbb{E}}

\newcommand{\bbG}{\mathbb{G}}

\newcommand{\bbL}{\mathbb{L}}

\newcommand{\bbN}{\mathbb{N}}

\newcommand{\bbP}{\mathbb{P}}

\newcommand{\bbR}{\mathbb{R}}
\newcommand{\bbS}{\mathbb{S}}

\newcommand{\bbU}{\mathbb{U}}
\newcommand{\bbV}{\mathbb{V}}

\newcommand{\bbZ}{\mathbb{Z}}

\newcommand{\bfS}{\mathbf{S}}
\newcommand{\bfT}{\mathbf{T}}

\newcommand{\sfC}{\mathsf{C}}

\newcommand{\eps}{\varepsilon}

\newcommand{\La}{\Lambda}
\newcommand{\om}{\omega}

\renewcommand{\abs}[1]{\left\vert #1 \right\vert}

\newcommand{\ind}{\boldsymbol 1}

\newcommand{\lr}[1][]{\xleftrightarrow{\: #1 \:\:}}

\begin{document}
	\title{The Wulff crystal of self-dual FK-percolation becomes round when approaching criticality}

	\author{Ioan Manolescu$^{\dagger}$, Maran Mohanarangan$^{\ast}$}
	
	\institute{
		$^\dagger$ University of Fribourg, \url{ioan.manolescu@unifr.ch} \newline $^\ast$ University of Fribourg \& University of Innsbruck,  \url{maran.mohanarangan@unifr.ch}
		\vspace{5pt}
		}

	\date{\today}
	
	\maketitle
	
	\begin{abstract}
		The study of the phase transition in planar FK-percolation on the square lattice has seen significant recent breakthroughs. The model undergoes a change in the nature of its phase transition at $q = 4$, transitioning from a continuous to a discontinuous regime. The aim of this article is to investigate the behaviour of the model in the discontinuous regime as $q > 4$ approaches the continuous transition point $4$ from above, while maintaining the critical parameter $p = p_c(q)$. We prove that in this limit, the correlation length becomes isotropic. The core of the proof builds upon the recently established rotational invariance of the large-scale features of the model at $q = 4$ \cite{DCKozKraManOulRotationalInvarianceCritPlanar}.
	\end{abstract}

	\tableofcontents
	
	\section{Introduction}
	
	\subsection{Definition of the model and main result}
	The \emph{random-cluster model}, introduced by Fortuin and Kasteleyn and hence also referred to as \emph{Fortuin--Kasteleyn (FK) percolation}, 
	is a central model of statistical mechanics. We consider it here on the two-dimensional square lattice $\bbZ^2$, where it exhibits a first-order or higher-order phase transition, depending on the cluster-weight $q$. 
	We recall its definition and a few basic properties below. For more background, we refer the reader to the monograph \cite{GrimmettRandomCluster} and, for an exposition of more recent results, to the lecture notes \cite{ManFKLectureNotes}.
	
	We consider the square lattice $(\bbZ^2,\bbE)$, that is, the graph with vertex set $\bbZ^2 = \{(n,m): n,m \in \bbZ^2\}$ and edges between nearest neighbours. We will slightly abuse notation and denote the graph itself by $\bbZ^2$. Let $G = (V,E)$ be a subgraph of the square lattice. A percolation configuration $\om$ on $G$ is an element of $\{0,1\}^E$. An edge $e \in E$ is said to be \emph{open} in $\om$ if $\om_e = 1$ and \emph{closed} otherwise. A configuration $\om$ can be seen as a subgraph of $G$ with vertex set $V$ and edge set $\{e \in E: \om_e = 1\}$. When speaking of \emph{connections} in $\om$, we view $\om$ as a graph. A \emph{cluster} is a maximal connected component of $\om$ (it may be an isolated vertex). When $G$ is finite, let $o(\om)$ and $c(\om)$ denote the number of open and closed edges in $\om$, respectively. Furthermore, let $k(\om)$ denote the number of clusters in $\om$. Then the random-cluster measure on $G$ with parameters $p \in [0,1]$ and $q > 0$, and free boundary conditions is a measure on percolation configurations given by
	\begin{align}
		\phi^0_{G,p,q}[\om] = \frac{1}{Z^{0}_{G,p,q}}p^{o(\om)}(1-p)^{c(\om)}q^{k(\om)},
	\end{align}
	where $Z^{0}_{G,p,q}$ is a normalising constant called the \emph{partition function}.
	
	For $q \geq 1$, the model can be extended to infinite volume where it exhibits a phase transition. More precisely, the family of measures $\phi^0_{G,p,q}$ converges weakly as $G$ tends to $\bbZ^2$. The limiting measure, denoted by $\phi^0_{\bbZ^2,p,q}$, undergoes a phase transition in the sense that there exists a critical threshold $p_c \coloneqq p_c(q) \in (0,1)$ such that the probability that there exists an infinite cluster under $\phi^0_{\bbZ^2,p,q}$ is 0 for all $p < p_c$ and 1 if $p > p_c$. It has been shown \cite{BeffaraDCSelfDualPointCritical} that the critical threshold is $p_c = \sqrt{q}/(1+\sqrt{q})$.
	
	The behaviour at the critical parameter $p_c$ is of great interest. One important question is whether the phase transition is \emph{continuous} or \emph{discontinuous}, meaning here whether the probability that $0$ is contained in an infinite cluster is continuous or discontinuous as a function of $p$. The phase transition was shown to be continuous for $1 \leq q \leq 4$ \cite{DCSidoraviciusTassionContPT} 
	and discontinuous when $q > 4$ \cite{DCGagnebinHarelManoTassionDiscontPT}. Alternative proofs of these two regimes were obtained in \cite{GlaLamDelocalisation} and \cite{RaySpinkaDiscontPT}, respectively.
	
	The behaviour at criticality differs drastically between the two regimes. In the regime $1 \leq q \leq 4$, the measure at the point of phase transition is unique and exhibits properties indicative of scale invariance, such as the Russo--Seymour--Welsh property \eqref{eq:RSW}; it is expected to have a non-trivial, conformally invariant scaling limit. 
	In contrast, for $q > 4$, there are multiple measures at the point of phase transition, with either sub- or super-critical behaviour, but no critical measure exists~\cite{GlaManGibbsFKPotts}. 
	In particular, in the first case the correlation length of the model diverges when approaching the critical point, while in the second case it remains bounded. 
	
	It was recently proved in \cite{DCKozKraManOulRotationalInvarianceCritPlanar} that for $1 \leq q \leq 4$, the model exhibits asymptotic rotational invariance at criticality, 
	in line with its conjectured conformal invariance. 
	More precisely, the model and its rotation by an arbitrary angle $\theta$ may be coupled so that, at large scales, they produce similar configurations with high probability. 
	In particular, any subsequential scaling limit of the critical model is invariant under all rotations. 
	This result will be instrumental in the proof of our main theorem.

	In this paper, we will be particularly interested in the discontinuous phase transition, $q > 4$. 
	For the exposition of our results, we need to introduce two central quantities. 
	Fix $\theta \in [0,2\pi)$ and define $e_\theta$ as the unit vector with angle $\theta$ to the horizontal axis. 
	For $q \geq1$ and $p \leq p_c$, the \emph{correlation length} of the model in direction $\theta$ is defined as the limit
	\begin{align}
		\xi_{p,q}(\theta) \coloneqq \Big(\lim_{n \to \infty} - \tfrac1n \log \phi^0_{\bbZ^2,p,q}\big[ 0 \lr \lfloor n e_\theta \rfloor\big]\Big)^{-1},
	\end{align}
	where $\lfloor n e_\theta \rfloor$ is the vertex of $\bbZ^2$ closest to $n e_\theta \in \bbR^2$. 
	Similarly, we define the \emph{point-to-hyperplane decay rate} by
	\begin{align}
		\zeta_{p,q}(\theta) \coloneqq \Big(\lim_{n \to \infty} -\tfrac1n \log\phi^0_{\bbZ^2,p,q}\big[0 \lr \calH_{\geq n}^\theta\big]\Big)^{-1},
	\end{align}
	where $\calH_{\geq n}^\theta$ is the half-space defined by
	\begin{align}
		\calH_{\geq n}^\theta \coloneqq \{x \in \bbR^2: \langle x,e_\theta \rangle \geq n\} \quad \text{ for $n\geq 0$}.
	\end{align}
	The existence of both limits follows from standard sub-additivity arguments.
	The two quantities are related via a so-called \emph{convex duality} (see e.g.\ \cite{OttDecayRates}). More precisely, it holds that
	\begin{align}\label{eq:convex-duality}
		\big(\xi_{p,q}(\theta)\big)^{-1} = \sup_{\theta' \in [0,2\pi)} \big(\zeta_{p,q}(\theta')\big)^{-1} \langle e_{\theta},e_{\theta'}\rangle.
	\end{align}

	A hallmark of the regime $q > 4$ is the finiteness of $\xi_{p_c,q}(\theta)$ and $\zeta_{p_c,q}(\theta)$ for all $\theta \in [0,2\pi)$. 
	It should be mentioned that $\xi_{p,q}(\theta)$ and $\zeta_{p,q}(\theta)$ are both finite for all $p< p_c$ and $q \geq 1$ \cite{BeffaraDCSelfDualPointCritical,DumRaoTasSharpnessFK}. 
	However, when $p = p_c$ and $q \in [1,4]$, both quantities are infinite \cite{DCSidoraviciusTassionContPT}. 
	
	The asymptotic rotational invariance of the model in the regime $1 \leq q \leq 4$ is a strong indication that $\xi_{p_c,q}$ and $\zeta_{p_c,q}$ become isotropic as $q \searrow 4$. 
	Our main result confirms this. Note that while both $\xi_{p_c,q}(\theta)$ and $\zeta_{p_c,q}(\theta)$ diverge for all $\theta$ as $q \searrow 4$, our result states that these quantities become asymptotically isotropic when renormalised.
	
	\begin{thm}\label{thm:rotational_invariance}
		For all $\eps >0$, there exists $q_0 > 4$ such that for all $q \in (4,q_0]$ and any $\theta_1, \theta_2 \in [0,2\pi)$, we have
		\begin{align}\label{eq:rot inv}
			\bigg|\frac{\xi_{p_c,q}(\theta_1) }{\xi_{p_c,q}(\theta_2)} -1\bigg| < \eps 
			\quad \text{and} \quad 
			\bigg|\frac{\zeta_{p_c,q}(\theta_1) }{\zeta_{p_c,q}(\theta_2)} -1\bigg| < \eps.
		\end{align}
	\end{thm}
	
	Theorem \ref{thm:rotational_invariance} may also be expressed in terms of the \emph{Wulff shape}, the polar set of the inverse correlation length:
	\begin{align}\label{eq:wulff_def}
		\calW_q \coloneqq \bigcap_{\theta \in [0,2\pi)} \{x \in \bbR^2 : \langle x,e_\theta \rangle \leq \xi_{p_c,q}(\theta)^{-1}\}.
	\end{align}
	Then Theorem \ref{thm:rotational_invariance} implies the following result.
	\begin{cor}\label{cor: wulff limit}
		When $q \searrow 4$, $\calW_q/\sqrt{{\rm Vol}(\calW_q)}$ tends to the unit disk $\bbU = \{x \in \bbR^2: \abs{x} \leq 1\}$.
	\end{cor}

	The Wulff shape describes the asymptotic shape of a cluster when conditioned to have a large volume. 
	Indeed, write $\sfC$ for the cluster of the origin and denote its cardinality by $\abs{\sfC}$. Then, for any $\eps > 0$,
	\begin{align}
		\phi_{\bbZ^2,p_c,q} \bigg[ {\rm d}_{\rm Hausdorff} \Big(\frac{1}{\sqrt{n}} \cdot \sfC, \frac{1}{\sqrt{{\rm Vol}(\calW_q)}} \cdot \calW_q\Big) > \eps \, \Big| \, \abs{\sfC} \geq n \bigg] \to 0
	\end{align}
	as $n \to \infty$, where $ {\rm d}_{\rm Hausdorff}$ denotes the Hausdorff distance \cite{CerfWulff}.
	Thus, Corollary~\ref{cor: wulff limit} states that, as $q \searrow 4$, the typical shape of a cluster conditioned to be large becomes round.
	
	\subsection{Strategy of the proof}
		
	Similar to the rotational invariance proved in \cite{DCKozKraManOulRotationalInvarianceCritPlanar}, our result is accompanied by a universality result relating random-cluster models on different isoradial graphs.
	To that end, we briefly describe an inhomogeneous FK-percolation model on some distorted embedding of the square lattice $\bbZ^2$. A more detailed discussion of this topic will be postponed until Section~\ref{sec: FK on isoradial graphs}.
	
	Fix $q > 4$ and $\alpha \in (0,\pi)$. Let $\bbL(\alpha)$ be the embedding of $\bbZ^2$ in which horizontal edges have length $2 \cos (\alpha/2)$, vertical edges have length $2 \sin (\alpha/2)$, and which is rotated by an angle of $\alpha/2$ --- see Figure~\ref{fig:isoradial_rectangular}. Define the inhomogeneous random-cluster model on finite subgraphs of $\bbL(\alpha)$ with cluster weight $q$ and edge-parameters $p_{\rm h}$ and $p_{\rm v}$ for the horizontal and vertical edges, respectively, given by
	\begin{align}
		\frac{p_{\rm h}}{1-p_{\rm h}} =
			 \frac{1-p_{\rm v}}{p_{\rm v}} = 
			 \sqrt{q} \frac{\sinh(r\alpha)}{\sinh(r(\pi-\alpha))}  \quad \text{with} \quad r \coloneqq \tfrac{1}{\pi} \cosh^{-1}\big(\frac{\sqrt{q}}{2}\big),
	\end{align}
	in the same way as the model on $\bbZ^2$; we also refer to \eqref{eq:def_iso_model} for a more general definition. 
	
	Write $\phi^0_{\bbL(\alpha),q}$ for the infinite-volume measure on $\bbL(\alpha)$ with the edge weights $p_{\rm h}$ and $p_{\rm v}$ as above and free boundary conditions. We call $\bbL(\alpha)$ an \emph{isoradial rectangular lattice} and $\phi_{\bbL(\alpha),q}^0$ its associated random-cluster model.
	
	Below, we will use estimates on lattices $\bbL(\alpha)$ for different values of $\alpha \in (0,\pi)$. 
	We will use the phrase ``uniform in $\alpha$ on compacts of $(0, \pi)$'' to mean that estimates are uniform on any interval of the type $(\eps, \pi-\eps)$ with $\eps > 0$, but potentially not on $(0,\pi)$. 
	All constants hereafter will be uniform in $\alpha$ on compacts of $(0, \pi)$. The main results may be shown to be uniform over the whole interval $(0,\pi)$, but this requires more care and is not essential for our purposes; for further details, see \cite[Sec.\ 5]{DCLiManoUniversality}.
	
	It was shown in \cite{DCLiManoUniversality} that, for any $q \geq 1$ fixed, the random-cluster models exhibit the same type of phase transition across all isoradial rectangular lattices $\bbL(\alpha)$. 
	In particular, as for the critical random-cluster model on $\bbZ^2$, the phase transition is discontinuous for $q > 4$. Thus, for any direction $\theta \in [0,2\pi)$ and any angle $\alpha \in (0,\pi)$, 
	the correlation length and the point-to-hyperplane decay rate may be defined and are bounded, uniformly in $\theta \in [0,2\pi)$ and $\alpha$ on compacts of $(0, \pi)$:
	\begin{align}
		\begin{aligned}
			\xi_{\alpha,q}(\theta) &=  \Big(\lim_{n \to \infty} - \tfrac{1}{n} \log \phi^0_{\bbL({\alpha}),q}\big[0 \lr \lfloor n e_\theta \rfloor\big]\Big)^{-1} <\infty  \quad\text{and}\\
			\zeta_{\alpha,q}(\theta) &= \Big(\lim_{n \to \infty} - \tfrac{1}{n}\log \phi^0_{\bbL(\alpha),q}\big[0 \lr \calH_{\geq n}^{\theta}\big]\Big)^{-1}<\infty.
			\label{eq:exponential_decay_L(alpha)}
		\end{aligned}
	\end{align}
	
	Moreover, while the convex duality in \cite{OttDecayRates} is stated for the square lattice $\bbZ^2$, the same proof still applies to isoradial rectangular lattices $\bbL(\alpha)$, i.e., the relation
	\begin{align}\label{eq:convex-duality_iso}
		\big(\xi_{\alpha,q}(\theta)\big)^{-1} = \sup_{\theta' \in [0,2\pi)} \big(\zeta_{\alpha,q}(\theta')\big)^{-1} \langle e_{\theta},e_{\theta'}\rangle.
	\end{align}
	still holds.
	Lastly, notice that for $\alpha = \tfrac\pi2$, one obtains a slight modification\footnote{Indeed, $\bbL(\pi/2)$ is the rotation by $\pi/4$ of $\sqrt 2 \bbZ^2$ and the measure $\phi^0_{\bbL(\alpha),q}$ is the free measure with $p = p_c(q)$ on this graph.} of $\xi_{p_c,q}$ and $\zeta_{p_c,q}$ as previously defined.
	
	The models $\phi^0_{\bbL(\alpha),q}$ for different values of $\alpha$ may be related by modifying the lattice step by step via the so-called \emph{star-triangle transformation}. Features that are stable under these transformations can then be transferred from one measure to the other. This strategy was used successfully for FK-percolation in \cite{GriManUniversalityBondPerco, DCLiManoUniversality, DCKozKraManOulRotationalInvarianceCritPlanar} to transfer RSW estimates and prove universality of the models for $q \in [1,4]$ and will be used here to obtain asymptotic universality for our quantities of interest.
	
	\begin{thm}\label{thm:universality}
		For all $\eps > 0$, there exists $q_0 > 4$ such that for $q \in (4, q_0]$, all $\alpha \in (\eps,\pi - \eps)$ and any $\theta \in [0,2\pi)$, it holds that
		\begin{align}
			\bigg|\frac{\xi_{\alpha,q}(\theta)}{\xi_{\pi/2,q}(\theta)}-1\bigg| < \eps 
			\quad \text{and} \quad
			\bigg|\frac{\zeta_{\alpha,q}(\theta)}{\zeta_{\pi/2,q}(\theta)}-1\bigg| < \eps.
		\end{align}
	\end{thm}
	
	We will focus on proving the statement for the point-to-hyperplane decay rate $\zeta$. 
	Indeed, this quantity is simpler to control under the local modifications we apply to the model. 
	The analogous statement for the correlation length $\xi$ will then be obtained via the convex duality \eqref{eq:convex-duality_iso}.
	
	We close this section by observing that Theorem~\ref{thm:universality} directly implies Theorem~\ref{thm:rotational_invariance}.
	
	\begin{proof}[Proof of Theorem~\ref{thm:rotational_invariance}]
		We prove the statement for the correlation length $\xi$ --- the analogous result for $\zeta$ is obtained in the same way.
		
		Fix $\eps > 0$ and $0 < \eps' < \frac{\eps}{2+\eps}$. Choose $q_0$ as in Theorem \ref{thm:universality} according to $\eps'$ and let $q \in (4,q_0]$. Fix $\theta_1 \in (0,\tfrac\pi4]$ and $\theta_2 \in [\tfrac\pi4,\tfrac\pi2]$, then set $\alpha = \theta_1+\theta_2 \in [\frac{\pi}4,\frac{3\pi}4]$. Observe that $e^{i\alpha/2} \bbR$ is an axis of symmetry of $\bbL(\alpha)$ and therefore, $\phi^0_{\bbL(\alpha),q}$ is invariant under orthogonal reflections with respect to said axis. In particular, we have $\xi_{\alpha,q}(\theta_1) = \xi_{\alpha,q}(\theta_2)$. Theorem \ref{thm:universality} then implies
		\begin{align}
			\frac{\xi_{\pi/2,q}(\theta_1)}{\xi_{\pi/2,q}(\theta_2)} 
			= \frac{\xi_{\pi/2,q}(\theta_1)/\xi_{\alpha,q}(\theta_1)}{\xi_{\pi/2,q}(\theta_2)/\xi_{\alpha,q}(\theta_2)} \in \bigg(\frac{1-\eps'}{1+\eps'},\frac{1+\eps'}{1-\eps'}\bigg).
		\end{align}
		In particular,
		\begin{align}
			\bigg|\frac{\xi_{\pi/2,q}(\theta_1)}{\xi_{\pi/2,q}(\theta_2)} - 1\bigg| < \frac{2\eps'}{1-\eps'}	<\eps.
		\end{align}
		
		The above extends to all angles $\theta_1, \theta_2 \in [0,2\pi)$ using the invariance of $\phi^0_{\bbL({\pi/2}),q}$ under reflections with respect to the horizontal, vertical, and diagonal axes. 
		Finally, note that $\xi_{p_c,q}(\theta) = \frac{1}{\sqrt 2}\xi_{\pi/2,q}(\theta + \frac{\pi}4)$. 
	\end{proof}
	
	\subsection{\texorpdfstring{Near-critical FK-percolation with $q \leq 4$}{Near-critical FK-percolation with q ≤ 4}}

	One may view the topic of the present paper in the context of near-critical FK-percolation on $\bbZ^2$.
	Most commonly, the near-critical regime refers to FK-percolation on $\bbZ^2$ with fixed $q \in [1,4]$ and $p\neq p_c$, at scales where the model transitions from a critical to an off-critical behaviour. 
	We direct the reader to \cite{DCManScalingRelations} for details
	and only mention here that the near-critical FK percolation with $1 \leq q \leq 4$ exhibits similar features as the critical one, most importantly \eqref{eq:RSW}. 

	It is expected that, as for the critical phase, the near-critical regime exhibits a form of asymptotic invariance under rotations. 
	In particular, we expect that, for any $1 \leq q \leq 4$ and any two angles $\theta_1,\theta_2$, 
	\begin{align}
		\frac{\xi_{p,q}(\theta_1) }{\xi_{p,q}(\theta_2)}  \to 1
		\quad \text{and} \quad 
		\frac{\zeta_{p,q}(\theta_1) }{\zeta_{p,q}(\theta_2)}  \to 1
		\qquad \text{ as $p\nearrow p_c(q)$}.
		\label{eq:conjectures_rot_inv_nc}
	\end{align}

	Theorem~\ref{thm:rotational_invariance} may be viewed as a manifestation of \eqref{eq:conjectures_rot_inv_nc} when the critical regime is approached along the line $q \searrow 4$ and $p  =p_c(q)$. 
	The authors have not managed to adapt the strategy below to also prove \eqref{eq:conjectures_rot_inv_nc} and we believe a key ingredient is missing. 
	Indeed, the {\em exact} validity of the star-triangle transformation (or its manifestation as track-exchanges; see Proposition~\ref{prop:track-exchange operator} below) is essential to the present argument. 
	These transformations are no longer exact when $p \neq p_c$, which we believe is a fundamental obstruction to adapting the argument. 

	It should be mentioned that \cite{DCWulff} proved \eqref{eq:conjectures_rot_inv_nc} for a variant of Bernoulli percolation (corresponding to $q=1$) through different means. 
	Indeed, this work builds on the existence and rotational invariance of the near-critical scaling limit \cite{GarPetSchNearCritScalingLimit}, which in turn relies on the description of the scaling limit of the critical phase \cite{SmiConformalInvariancePerco,CamNewScalingLimitPerco}.
	This approach has not yet been adapted to $q > 1$ and will likely require significant new ideas. 
	
	\subsection*{Organisation of the paper}
	Section~\ref{sec: FK on isoradial graphs} contains a review of some basic results for FK-percolation on isoradial graphs. In Section~\ref{sec: hp-measures}, we define half-plane measures and track-exchanges for such measures, as this framework is technically more convenient for our purposes. Finally, in Section~\ref{sec:universality}, we prove the main result, namely Theorem~\ref{thm:universality}.
	
	\section{Preliminaries}\label{sec: FK on isoradial graphs}
	
	In this section, we introduce isoradial graphs and the random-cluster models associated to them.
	
	\subsection{Isoradial graphs}
	
	A rhombic tiling $\bbG^\diamond$ is a tiling of the plane by rhombi of edge-length 1. Any such graph is bipartite and we can divide its vertices in two sets of non-adjacent vertices $\bbV_\bullet$ and $\bbV_\circ$. The \emph{isoradial graph} $\bbG$ associated to $\bbG^\diamond$ is the graph with vertex set $\bbV_\bullet$ and edge set given by the diagonals of the faces of $\bbG^\diamond$ between vertices of $\bbV_\bullet$. If the roles of $\bbV_\bullet$ and $\bbV_\circ$ are exchanged, we obtain the dual $\bbG^*$ of $\bbG$, which is also isoradial. The term isoradial refers to the fact that each face of $\bbG$ can be inscribed in a circle of radius 1. The rhombic tiling $\bbG^\diamond$ is called the \emph{diamond graph} of $\bbG$.
	
	A \emph{track} of $\bbG^\diamond$ is a bi-infinite sequence of adjacent faces $(r_i)_{i \in \bbZ}$ of $\bbG^\diamond$, such that each pair $r_i$ and $r_{i+1}$ shares an edge, and all such edges are parallel. The angle formed by any such edge with the horizontal axis is called the \emph{transverse angle} of the track. 
	
	The graphs considered in this paper are of a specific type: we assume that all faces of $\bbG^\diamond$ have horizontal top and bottom edges --- we call these {\em isoradial rectangular lattices}. As a result, the diamond graphs consist of two families of tracks: \emph{horizontal tracks} $t_i$ with transverse angles $\alpha_i \in (0,\pi)$, and \emph{vertical tracks} $s_j$, each with transverse angle 0. Each track of one family intersects all tracks of the other family, but no two tracks from the same family intersect.
	
	For a sequence of track angles $\bm \alpha = (\alpha_i)_{i \in \bbZ} \in (0,\pi)^{\bbZ}$, denote by $\bbL(\bm\alpha)$ the graph as above whose horizontal tracks have transverse angles $\alpha_i$ in increasing vertical order. When $\alpha_i = \alpha$ for every $i$, we simply write $\bbL(\alpha) = \bbL(\bm \alpha)$. 
	Note that $\mathbb{L}(\alpha)$ corresponds to a distorted embedding of the square lattice $\bbZ^2$,  with $e^{i \alpha/2} \bbR$ and $e^{i (\alpha+ \pi)/2} \bbR$ as axes of symmetry; see Figure~\ref{fig:isoradial_rectangular} for an illustration. 
	
	We will use the same notation for finite or half-infinite sequences $\bm\alpha$ to indicate lattices  $\bbL(\bm\alpha)$ covering a horizontal strip or half-plane.
	For technical reasons (specifically for the results of \cite{DCLiManoUniversality} to adapt readily), we will always work with sequences $\bm \alpha$ containing at most two values. 
	This restriction is not essential and may be removed by revisiting \cite{DCLiManoUniversality}.
	
	\begin{figure}
		\centering
		\small
		
	\def\svgwidth{.9\columnwidth}
\begingroup%
  \makeatletter%
  \providecommand\color[2][]{%
    \errmessage{(Inkscape) Color is used for the text in Inkscape, but the package 'color.sty' is not loaded}%
    \renewcommand\color[2][]{}%
  }%
  \providecommand\transparent[1]{%
    \errmessage{(Inkscape) Transparency is used (non-zero) for the text in Inkscape, but the package 'transparent.sty' is not loaded}%
    \renewcommand\transparent[1]{}%
  }%
  \providecommand\rotatebox[2]{#2}%
  \newcommand*\fsize{\dimexpr\f@size pt\relax}%
  \newcommand*\lineheight[1]{\fontsize{\fsize}{#1\fsize}\selectfont}%
  \ifx\svgwidth\undefined%
    \setlength{\unitlength}{313.3860911bp}%
    \ifx\svgscale\undefined%
      \relax%
    \else%
      \setlength{\unitlength}{\unitlength * \real{\svgscale}}%
    \fi%
  \else%
    \setlength{\unitlength}{\svgwidth}%
  \fi%
  \global\let\svgwidth\undefined%
  \global\let\svgscale\undefined%
  \makeatother%
  \begin{picture}(1,0.29693998)%
    \lineheight{1}%
    \setlength\tabcolsep{0pt}%
    \put(0,0){\includegraphics[width=\unitlength,page=1]{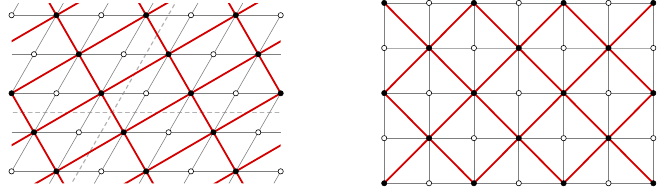}}%
    \put(0.1536281,0.25456263){\color[rgb]{0.50196078,0,0}\transparent{0.99848002}\makebox(0,0)[lt]{\lineheight{0}\smash{\begin{tabular}[t]{l}$p_{\rm h}$\end{tabular}}}}%
    \put(0.11058207,0.17808896){\color[rgb]{0.50196078,0,0}\transparent{0.99848002}\makebox(0,0)[lt]{\lineheight{0}\smash{\begin{tabular}[t]{l}$p_{\rm v}$\end{tabular}}}}%
    \put(-0.00075792,0.11738444){\color[rgb]{0,0,0}\transparent{0.99848002}\makebox(0,0)[lt]{\lineheight{0}\smash{\begin{tabular}[t]{l}$t_i$\end{tabular}}}}%
    \put(0.09848293,0.00593154){\color[rgb]{0,0,0}\transparent{0.99848002}\makebox(0,0)[lt]{\lineheight{0}\smash{\begin{tabular}[t]{l}$s_j$\end{tabular}}}}%
    \put(0,0){\includegraphics[width=\unitlength,page=2]{isoradial-lattices.pdf}}%
    \put(0.03197258,0.04415402){\color[rgb]{0,0,0}\transparent{0.99848002}\makebox(0,0)[lt]{\lineheight{0}\smash{\begin{tabular}[t]{l}$\alpha$\end{tabular}}}}%
  \end{picture}%
\endgroup%

		\caption{{\em Left:} An isoradial rectangular lattice $\bbL(\alpha)$. The underlying diamond graph is drawn in thin black lines, while the actual lattice is drawn in red. The white points denote vertices of the dual lattice. Two tracks --- one horizontal and one vertical --- are shown as dashed grey lines. Note the different parameters $p_{\rm v}$ and $p_{\rm h}$ for the two edge orientations.
		{\em Right:} For $\alpha = \pi/2$, the diamond graph is $\bbZ^2$ and the lattice $\bbL(\pi/2)$ is the rotation of $\sqrt{2} \bbZ^2$ by $\frac\pi4$. In this case, the model is homogenous, as $p_{\rm v} = p_{\rm h} = \sqrt{q}/(1+\sqrt{q})$.} 
		\label{fig:isoradial_rectangular}
	\end{figure}
	
	When considering isoradial graphs $\bbG=(\bbV,\bbE)$, we use the notation $\La_n= [-n,n]^2$ and identify it with the subgraph spanned by the vertices of $\bbV$ contained in $\La_n$. We write $\partial \La_n$ for the set of vertices $v \in \bbV \cap \La_n$ that have a neighbour in $\bbV \cap \La_n^c$. Furthermore, we write $\La_n(z)$ for the translation of $\La_n$ by $z \in \bbR^2$.
	
	\subsection{Definition and elementary properties of the isoradial random-cluster model}
	
	The isoradial embedding of a graph $\bbG$ produces different edge-weights --- called \emph{isoradial weights} --- for the edges of $\bbG$ as a function of their length. Indeed if $e$ is an edge of $\bbG$ and $\theta_e$ is the angle of the rhombus of $\bbG^\diamond$ containing $e$ and not bisected by $e$, we set
	
	\begin{align}
		p_e \coloneqq \begin{cases} \frac{\sqrt{q} \sin(r(\pi-\theta_e))}{\sin(r\theta_e)+\sqrt{q}\sin(r(\pi-\theta_e))}, &q < 4,\\ \frac{2\pi-2\theta_e}{2\pi-\theta_e}, &q=4,\\ \frac{\sqrt{q} \sinh(r(\pi-\theta_e))}{\sinh(r\theta_e)+\sqrt{q}\sinh(r(\pi-\theta_e))}, &q > 4, \end{cases} \quad \text{where} \quad r \coloneqq \begin{cases} \frac{1}{\pi} \cos^{-1}\left( \frac{\sqrt{q}}{2}\right), & q \leq 4,\\ \frac{1}{\pi} \cosh^{-1} \left(\frac{\sqrt{q}}{2} \right), &q >4. \end{cases}
	\end{align}

	For a finite subgraph $G=(V,E)$ of an isoradial graph $\bbG =(\bbV,\bbE)$, we define its vertex boundary $\partial G$ as the set of vertices in $V$ incident to a vertex in $\bbV \setminus V$. A \emph{boundary condition} $\xi$ on $G$ is then given by a partition of the set $\partial G$. We say that two vertices of $G$ are \emph{wired} if they belong to the same element of the partition $\xi$. When no boundary vertices are wired together, we obtain \emph{free} boundary conditions, which are denoted by 0.
	
	\begin{definition}
		For a finite subgraph $G = (V,E)$ of an isoradial graph $\bbG$, $q \geq 1$, and a boundary condition $\xi$, the \emph{random-cluster measure} on $G$ with cluster weight $q$ and boundary conditions $\xi$ is the measure $\phi^\xi_{G,q}$ on $\{0,1\}^E$ given by
		\begin{align}\label{eq:def_iso_model} 
			\phi^\xi_{G,q}[\om] = \frac{1}{Z^\xi_{G,q}} q^{k(\om^\xi)} \prod_{e \in E} p_e^{\om_e}(1-p_e)^{1-\om_e},
		\end{align}
		where the $p_e$ are the isoradial weights defined above, $k(\om^\xi)$ is the number of connected components in the graph $\om^\xi$ which is obtained from $\om$ by identifying vertices which are wired in $\xi$, and $Z^\xi_{G,q}$ is a normalising constant called the partition function.
	\end{definition}
	
	We will mostly consider the random-cluster model on infinite isoradial graphs $\bbG$ with free boundary conditions obtained by taking the limit of the measures with free boundary conditions on larger and larger finite graphs $G$ tending to $\bbG$. We write $\phi^0_{\bbG,q}$ for the measure in infinite volume that is obtained this way. Notice that the measures $\phi^0_{\bbL(\alpha),q}$ defined in the introduction are indeed those obtained by the isoradial weights above.
	
	Let us finally mention that the basic properties of the random-cluster model extend readily to this setting and we refer to \cite{GrimmettRandomCluster} for proofs in the homogenous setting. The following standard properties will be used repeatedly without mention.
	
	\paragraph{Monotonic properties.} Fix a finite subgraph $G = (V,E)$ of an isoradial graph $\bbG$, $q \geq 1$ and a boundary condition $\xi$. We say that an event $A$ is \emph{increasing} if for any $\om \leq \om'$\footnote{Here, $\om \leq \om'$ refers the partial ordering on $\{0,1\}^E$ given by $\om \leq \om'$ if $\om_e \leq \om'_e$ for every edge $e \in E$.}, $\om \in A$ implies $\om' \in A$. The \emph{FKG inequality} states that for any increasing events $A$ and $B$,
	\begin{align}\label{eq:FKG}
		\phi^\xi_{G,q}[A \cap B] \geq \phi^\xi_{G,q}[A] \phi^\xi_{G,q}[B]. \tag{FKG}
	\end{align}
	For two boundary conditions $\xi$ and $\xi'$ such that $\xi \leq \xi'$, meaning that any pair of vertices wired in $\xi$ is also wired in $\xi'$, the random-cluster model with boundary conditions $\xi$ is dominated by the one with boundary conditions $\xi'$, i.e., for any increasing event $A$,
	\begin{align}\label{eq:CBC}
		\phi^{\xi'}_{G,q}[A] \geq \phi^{\xi}_{G,q}[A]. \tag{CBC}
	\end{align}
	The latter inequality will be referred to as \emph{comparison between boundary conditions}.
	
	\paragraph{Spatial Markov property.} For $G$ as above, a subgraph $H$ of $G$, $q \geq 1$ and a boundary condition $\xi$ on $G$, we have
	\begin{align}\label{eq:SMP}
		\phi^\xi_{G,q}[ \om \text{ on } H \, | \, \om \text{ on } G \setminus H] = \phi^\zeta_{H,q}[\om \text{ on } H], \tag{SMP}
	\end{align}
	where $\zeta$ is the boundary condition induced by $\om^\xi$ on $G \setminus H$ --- that is, vertices of $\partial H$ are wired together in $\zeta$ if they are connected in $\omega^\xi$ in $G\setminus H$.
	
	A direct consequence of \eqref{eq:SMP} is the \emph{finite energy property}, meaning that the probability of an edge being open is bounded away from 0 and 1, uniformly in the state of all other edges. That is, for any edge $e$, there exists a constant $c_{\rm FE} > 0$ (only depending on the angle $\theta_e$) such that 
	\begin{align}\label{eq:FE}
		\phi^\xi_{G,q}[ \om_e = 1 \, | \, \om \text{ on } G \setminus \{e\}] \in [c_{\rm FE}, 1-c_{\rm FE}]. \tag{FE}
	\end{align}
	Moreover, when considering lattices $\bbL(\bm \alpha)$, $c_{\rm FE}$ is uniformly positive on compacts of $(0,\pi)$.
	
	\paragraph{The track-exchange operator.} 
	
	 Fix $q \geq 1$ and consider a finite sequence of angles  $\alpha_1,\dots, \alpha_n$ 
	 and let $\bbS$ denote the strip of an isoradial rectangular lattice with horizontal tracks $t_1,\dots, t_n$ with transverse angles   $\alpha_1,\dots, \alpha_n$
	 and vertical tracks of transverse angles $0$. 
	 Fix $ 1 < i \leq n$ and let $ \bbS'$ denote the strip obtained from the same sequence of angles, but with $\alpha_i$ and $\alpha_{i-1}$ exchanged. 
	 
	 The boundary vertices of $\bbS$ and $\bbS'$ are those below $t_1$ and above $t_n$; they are the same in the two strips. 
	 Let $\xi$ be a boundary condition for these two strips. 
	 The measures $\phi_{\bbS,q}^\xi$ and $\phi_{\bbS',q}^\xi$ may be defined as the unique limits of measures on $\bbS \cap\La_N$ as $N\to\infty$.
	 Their uniqueness is an immediate consequence of the finite energy property \eqref{eq:FE}. 
	 
	\begin{prop}\label{prop:track-exchange operator}
		There exists a random map $\bfT_i$ from percolation configurations on $\bbS$ to percolation configurations on $\bbS'$ such that
		\begin{itemize}
			\item if $\om \sim \phi^\xi_{\bbS,q}$, then $\bfT_i(\om) \sim \phi^\xi_{\bbS',q}$;
			\item $\om$ and $\bfT_i(\om)$ agree on all edges outside of $t_i$ and $t_{i-1}$;
			\item on edges of $t_i$ and $t_{i-1}$, $\bfT_i(\om)$ only depends on the values of $\omega$ on edges of $t_i$ and $t_{i-1}$ (and potentially on an additional source of randomness, independent of the rest of the configuration $\omega$);
			\item vertices outside of $t_i \cap t_{i-1}$ are connected in the same way in $\om$ and $\bfT_i(\om)$;
			\item the modification is local: there exists $c > 0$ such that, for $\om \sim \phi^\xi_{\bbS,q}$, 
			the state of $\bfT_i(\om)$ on any edge $e$ is determined\footnote{The transformation may use extra randomness, which may be encoded by i.i.d.\ uniform variables $U_e$ on $[0,1]$ for $e \in t_{i-1} \cup t_i$. The precise statement is that $(\bfT_i(\om))_e$ may be determined with probability exponentially close to $1$ from the knowledge of  $\omega$ and the variables $U_.$ on $\La_r(e)$.}
			by $\om$ restricted to $\La_r(e)$ with probability at least $1-e^{-cr}$.
		\end{itemize}
	\end{prop}
	
	The map $\bfT_i$ will be called the \emph{track-exchange operator}. The precise construction of the coupling between $\om$ and $\bfT_i(\om)$ is not essential, nor unique. One such coupling may be obtained by compositions of the star-triangle transformation; we refer the reader to \cite{DCKozKraManOulRotationalInvarianceCritPlanar} for details. We would merely like to remark that Proposition~\ref{prop:track-exchange operator} is an indication of the exact integrability of the model defined by the isoradial weights.
	We will also use $\bfT_i$ as a transformation of lattices, and write $\bbS' = \bfT_i(\bbS)$. 

	\subsection{\texorpdfstring{Isoradial random-cluster model with $q=4$}{Isoradial random-cluster model with q = 4}}\label{sec:background_q=4}
	
	Recall from \eqref{eq:exponential_decay_L(alpha)} that the random-cluster model on isoradial rectangular lattices with $q > 4$ shares the features of the model on $\bbZ^2$, 
	in particular its free infinite-volume measure exhibits exponential decay of connection probabilities. 
	Similarly, for $q=4$, the random-cluster measure on isoradial rectangular lattices behaves like the critical measure on $\bbZ^2$. We review its main characteristics here. 
	
	While these properties hold across the entire regime $q \in [1,4]$, our primary focus will be on the case $q=4$ and they will be stated as such. This specific choice is motivated by our ultimate goal to transfer some estimates from $q=4$ to the regime $q > 4$ with $q$ sufficiently close to $4$.
	
	In the following, let $\bbL$ be a lattice of the type $\bbL(\bm \alpha)$, where $\bm \alpha \in (0,\pi)^\bbZ$ is a sequence containing at most two values.	
	All constants below will be uniform in the values of $\bm \alpha$ on compacts of $(0,\pi)$. 
	It was proved in \cite{DCLiManoUniversality} that, for $q=4$, there exists a unique infinite-volume measure on $\bbL$, denoted by $\phi_{\bbL,4}$.
	
	\paragraph{Duality.} Recall that if $\bbG$ is an isoradial graph, then so is its dual $\bbG^*$. To each configuration $\om$ on $\bbG$, we associate a dual configuration $\om^*$ on $\bbG^*$ defined by taking $\om^*_{e^*} = 1-\om_e$ for all edges $e$ of $\bbG$, where $e^*$ is the unique edge of $\bbG^*$ intersecting $e$. The uniqueness of the infinite-volume measure on $\bbL$ implies that the dual of the isoradial measure on $\bbL$ is the isoradial measure on the dual lattice $\bbL^*$, i.e., if $\om \sim \phi_{\bbL,4}$, then $\om^* \sim \phi_{\bbL^*,4}$.
	
	\paragraph{RSW property.} Arguably the most useful ingredient in the study of critical planar percolation models is the \emph{Russo--Seymour--Welsh (RSW)} estimate. It states that the probabilities of crossing rectangles of a given aspect ratio but arbitrary scale are uniformly bounded away from $0$ and $1$. Moreover, these bounds are uniform, even when boundary conditions are imposed at a macroscopic distance from the rectangle. More precisely, for $\rho, \eps > 0$, there exists $c = c(\rho, \eps) > 0$ such that for any event $A$ depending on the edges at distance at least $\eps n$ from the rectangle $[0,\rho n] \times [0,n]$,
	\begin{align}\label{eq:RSW}
		c \leq \phi_{\bbL,4}\big[\scrC([0, \rho n] \times [0,n]) \,|\, A \big] \leq 1-c, \tag{RSW}
	\end{align}
	where $\scrC([0,\rho n] \times [0,n])$ denotes the event that there exists a path of open edges in $[0, \rho n] \times [0,n]$ from $\{0\} \times [0,n]$ to $\{\rho n\} \times [0,n]$. We refer to \cite[Thm.\ 1.1]{DCLiManoUniversality} for a full proof.
	
	We now discuss some consequences of \eqref{eq:RSW}. 
	
	\paragraph{Mixing.} One consequence we will use repeatedly is the so-called mixing property. For every $\eps > 0$, there exist $c_{\rm mix}, C_{\rm mix} \in (0,\infty)$ such that for every $r \leq R/2$, every event $A$ depending on edges in $\La_r$, and every event $B$ depending on edges outside $\La_R$, we have that
	\begin{align}\label{eq:mixing}
		\big|\phi_{\bbL,4}[A \cap B]- \phi_{\bbL,4}[A]\phi_{\bbL,4}[B]\big| \leq C_{\rm mix} \big(\tfrac{r}{R}\big)^{c_{\rm mix}} \phi_{\bbL,4}[A]\phi_{\bbL,4}[B]. \tag{Mix}
	\end{align}
	
	This property can be derived from \eqref{eq:RSW} using the same arguments as in the homogeneous model, see e.g.\ \cite[Cor.\ 2.10]{DCManScalingRelations}.
	
	\paragraph{Arm events and flower domains.}
	
	Fix some angle $\theta \in [0,2\pi)$ and define the following arm events. For $r \leq R$ and $z \in \bbL$, let $A_3^{{\rm hp}(\theta)} (z; r,R)$ be the event that there exist three non-intersecting paths 
	$\gamma_1,\gamma_2,\gamma_3$ in the annulus $\La_R(z) \setminus \La_r(z)$, between $\partial \La_r(z)$ and $\partial \La_R(z)$, 
	all contained in the half-space $\langle \cdot,e_{\theta} \rangle \leq \langle z ,e_{\theta} \rangle$, arranged in clockwise order and such that $\gamma_1,\gamma_3 \in \omega^*$ and $\gamma_2 \in \omega$. 
	We call this the three-arm event in the half-plane orthogonal to $e_\theta$ and call the paths $\gamma_1,\gamma_2,\gamma_3$ arms.
	
	We will now briefly introduce the notion of \emph{flower domains}. These domains are particularly suitable for studying arm events. We refer the reader to \cite{DCManScalingRelations} for a more comprehensive treatment.
	
	Given $R \geq 1$ and a configuration $\om$, let $\calE$ be the union of $\La_{2R}^c$ and all the primal/dual interfaces starting on $\partial \La_{2R}$ explored inwards until they either exit $\La_{2R}$ or enter $\La_R$. Let $\calF$ be the connected component of the origin in $\calE^c$. The boundary of $\calF$ is formed either of a primal or dual circuit, or of an even number of alternating primal and dual arcs called \emph{petals}. We call $\calF$ the flower domain revealed from $\La_{2R}$ to $\La_{R}$.
	See Figure~\ref{fig: good-scale} for an example. 
	
	For $\eta >0$, the flower domain $\calF$ is said to be $\eta$-\emph{well-separated} if the endpoints of its petals are at a distance of at least $\eta R$ from each other.
	In particular, a flower domain with a single petal is $\eta$-well-separated by default.

	It was proved in~\cite[Lemma 3.2]{DCManScalingRelations} that
	the flower domain from $\La_{2R}$ to $\La_{R}$ is well-separated with positive probability under $\phi^\xi_{\bbL \cap \La_{2R}}$ for any boundary condition $\xi$. 
	This can be extended to accommodate measures conditioned on arm events \cite{GasManMohArmSeparation}. 
	
	Let us give a more precise statement in the case of the three-arm event in the half plane.
	We say that a flower domain is \emph{good} if it has exactly two petals (one primal and one dual), is $\frac14$-well-separated, and its primal petal is contained in the cone with apex $0$, bisector $-e_\theta$, and aperture $\pi/2$ (see Figure~\ref{fig: good-scale}).
	
	\begin{figure}
		\centering
		\includegraphics[width = 0.75\textwidth]{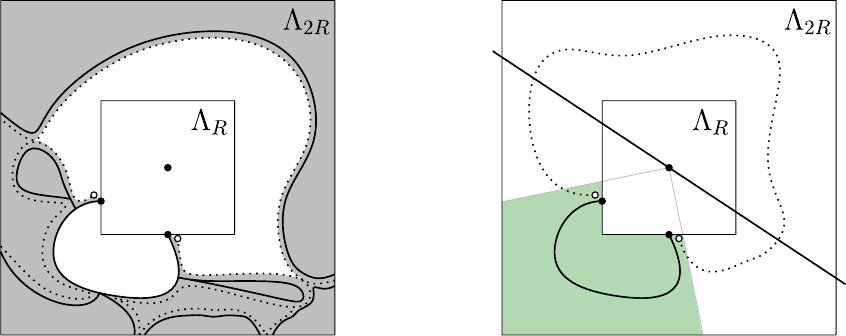}
		\caption{{\em Left:} the flower domain from $\La_{2R}$ to $\La_R$ is obtained by exploring all primal--dual interfaces starting on $\partial \La_{2R}$ until they exit $\La_{2R}$ or reach $\La_R$. The explored region is grey and the flower domain is the white domain; its boundary is formed of one primal and one dual petal.  
		{\em Right:} a good flower domain with the unique primal petal contained in the green region.}
		\label{fig: good-scale}
	\end{figure}
	
	\begin{lem}[\cite{GasManMohArmSeparation}]\label{lem:good flower domain}
		There exists $c > 0$ such that for every $4r \leq R$, any $z \in \bbL$, and any configuration $\om_0$ on $\La_{4r}(z)^c$ that allows for the occurrence of $A_3^{{\rm hp}(\theta)}(z;1,R)$, we have
		\begin{align}
			\phi_{\bbL,4}\big[\calF\text{ is good}\,|\, \om = \om_0 \text{ on } \La_{4r}(z)^c \text{ and } A_3^{{\rm hp}(\theta)}(z;1,R)\big] \geq c,
		\end{align}
		where $\calF$ is the flower domain revealed from $\La_{2r}(z)$ to $\La_{r}(z)$.
	\end{lem}

	\paragraph{Arm exponents.}
	Another classical consequence of \eqref{eq:RSW} is that the probabilities of arm events may be bounded by polynomials with strictly positive exponents.
	
	The following arm exponent bounds will be useful in our arguments. 
	We call $\bbL = \bbL(\bm\alpha)$ \emph{periodic} if $\bm\alpha$ is a periodic sequence.
	
	\begin{prop}\label{prop:arm_exp}
		There exist constants $c,C > 0$ independent of $\theta$ and $\bbL$ such that, for all $1 \leq r<R$ and $z \in \bbL$, 
		\begin{align}
			\tfrac{1}{C}\big(\tfrac{r}{R}\big)^2 
			\leq \phi_{\bbL,4} \big[A_3^{{\rm hp}(\theta)} (z; r,R)\big]   
			&\leq C \big(\tfrac{r}{R}\big)^2,  &&\text{ if $\bbL$ is periodic, and}\label{eq:periodic 3-arm}\\
			\phi_{\bbL,4} \big[A_3^{{\rm hp}(\theta)} (z; r,R)\big]  &\leq C \big(\tfrac{r}{R}\big)^{1+c} &&\text{ for all $\bbL$ as above.}\label{eq:generic 3-arm}
		\end{align}
		In the first case, $C$ may depend on $\bbL$, but is uniform for sequences $\bm\alpha$ of period $2$ with values in compacts of $(0,\pi)$. 
		
		Moreover, the arm event probabilities satisfy a \emph{quasi-multiplicativity} property, i.e., for any $r \leq \rho \leq R$, we have
		\begin{align}\label{eq:quasi_multiplicativity}
			\tfrac{1}C 
			\leq \frac{\phi_{\bbL,4} \big[A_3^{{\rm hp}(\theta)} (z; r,R)\big]}{ \phi_{\bbL,4} \big[A_3^{{\rm hp}(\theta)} (z; r,\rho)\big]\phi_{\bbL,4} \big[A_3^{{\rm hp}(\theta)} (z; \rho,R)\big] } 
			\leq C.
		\end{align}
	\end{prop}
	
	We refer to \cite[Prop.\ 3.4]{DCKozKraManOulRotationalInvarianceCritPlanar} for a proof of \eqref{eq:periodic 3-arm}, \eqref{eq:generic 3-arm} and merely note that they are consequences of \eqref{eq:RSW} and certain symmetries of the lattice. Note that \eqref{eq:periodic 3-arm} requires the invariance of $\bbL$ under two independent translations, without which, only the weaker bound~\eqref{eq:generic 3-arm} may be obtained\footnote{Using the universality result of \cite{DCKozKraManOulRotationalInvarianceCritPlanar}, \eqref{eq:periodic 3-arm} may be extended to all $\bbL$ as above, but this requires additional work and is not necessary.}. The quasi-multiplicativity \eqref{eq:quasi_multiplicativity} is a standard consequence of \eqref{eq:RSW} and the arm-separation property; see \cite{CheDCHonCrossingProbaFKIsing,GasManMohArmSeparation} for details.

	\subsection{Incipient infinite cluster in the half-plane}\label{sec: IIC}
	
	In this section, we introduce the \emph{incipient infinite cluster} (henceforth IIC) measures with three arms in the half-plane.
	Fix two angles $\alpha,\beta \in (0,\pi)$ and write $\bbL_{\rm mix}$ for the lattice of type $\bbL(\bm \alpha)$ with $\bm \alpha \in (0,\pi)^\bbZ$ alternating between $\alpha$ and $\beta$, i.e., $\bm \alpha = (\alpha_i)_{i \in \bbZ}$ with $\alpha_i = \beta$ for $i$ even and $\alpha_i = \alpha$ for $i$ odd. Note that the \eqref{eq:RSW} property applies to $\bbL_{\rm mix}$.
	
	Recall that the horizontal tracks of $\bbL_{\rm mix}$ are denoted by $(t_k)_{k \in \bbZ}$ and the vertical ones by $(s_k)_{k \in \bbZ}$. We assume the vertex between $t_0$ and $t_1$ and between $s_0$ and $s_1$ to be a primal one and consider it to be the origin of $\bbR^2$. For integers $i,j$, define the \emph{cell} $(i,j)$ as the set of primal and dual vertices lying between the vertical tracks $s_{2i-1}$ and $s_{2i+1}$ and the horizontal tracks $t_{2j-1}$ and $t_{2j+1}$. Note that the cells are centred around rhombi of angle $\beta$. To each cell, we associate its lower-left lattice point, which, by our convention, is a primal vertex --- see Figure~\ref{fig: extremal cell}.
	
	Fix $\theta \in [0,2\pi)$. For a finite cluster $\sfC$ of a configuration $\om$ on $\bbL_{\rm mix}$, 
	let ${\rm Ext}_\theta(\sfC)$ be the lattice point that maximises the scalar product with $e_\theta$ and whose associated cell intersects $\sfC$. 
	If multiple such maximisers exist, choose the one with the largest vertical coordinate. 
	We call ${\rm Ext}_\theta(\sfC)$ the {\em extremum in direction $e_\theta$} of $\sfC$ --- note that it is possible for ${\rm Ext}_\theta(\sfC)$ to not be part of $\sfC$.
	Finally, let 
	\begin{align} 
		{\rm E}_\theta(\sfC) = \langle {\rm Ext}_\theta(\sfC), e_\theta \rangle
	\end{align}
	be the corresponding extremal coordinate of $\sfC$ in direction $e_\theta$.  See Figure~\ref{fig: extremal cell}.
	
	\begin{figure}
		\centering
		
	\def\svgwidth{.7\columnwidth}
\begingroup%
  \makeatletter%
  \providecommand\color[2][]{%
    \errmessage{(Inkscape) Color is used for the text in Inkscape, but the package 'color.sty' is not loaded}%
    \renewcommand\color[2][]{}%
  }%
  \providecommand\transparent[1]{%
    \errmessage{(Inkscape) Transparency is used (non-zero) for the text in Inkscape, but the package 'transparent.sty' is not loaded}%
    \renewcommand\transparent[1]{}%
  }%
  \providecommand\rotatebox[2]{#2}%
  \newcommand*\fsize{\dimexpr\f@size pt\relax}%
  \newcommand*\lineheight[1]{\fontsize{\fsize}{#1\fsize}\selectfont}%
  \ifx\svgwidth\undefined%
    \setlength{\unitlength}{299.78856365bp}%
    \ifx\svgscale\undefined%
      \relax%
    \else%
      \setlength{\unitlength}{\unitlength * \real{\svgscale}}%
    \fi%
  \else%
    \setlength{\unitlength}{\svgwidth}%
  \fi%
  \global\let\svgwidth\undefined%
  \global\let\svgscale\undefined%
  \makeatother%
  \begin{picture}(1,0.52596719)%
    \lineheight{1}%
    \setlength\tabcolsep{0pt}%
    \put(0,0){\includegraphics[width=\unitlength,page=1]{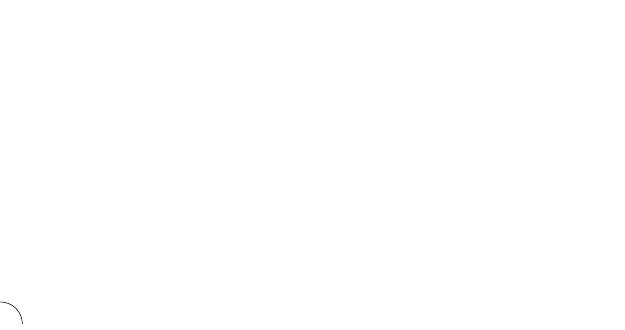}}%
    \put(0.00556582,0.02092893){\color[rgb]{0,0,0}\makebox(0,0)[lt]{\lineheight{0}\smash{\begin{tabular}[t]{l}$_\alpha$\end{tabular}}}}%
    \put(0,0){\includegraphics[width=\unitlength,page=2]{IIC.pdf}}%
    \put(0.01743497,0.07266356){\color[rgb]{0,0,0}\makebox(0,0)[lt]{\lineheight{0}\smash{\begin{tabular}[t]{l}$_\beta$\end{tabular}}}}%
    \put(0,0){\includegraphics[width=\unitlength,page=3]{IIC.pdf}}%
    \put(0.43785462,0.50286487){\color[rgb]{0,0,0}\makebox(0,0)[t]{\lineheight{1.25}\smash{\begin{tabular}[t]{c}$e_{\theta}$\end{tabular}}}}%
    \put(0,0){\includegraphics[width=\unitlength,page=4]{IIC.pdf}}%
    \put(0.18868645,0.3609236){\color[rgb]{0,0,0}\makebox(0,0)[t]{\lineheight{1.25}\smash{\begin{tabular}[t]{c}${\rm Ext}_\theta(\sfC)$\end{tabular}}}}%
    \put(0,0){\includegraphics[width=\unitlength,page=5]{IIC.pdf}}%
  \end{picture}%
\endgroup%

		\caption{A cluster $\sfC$ with its extremum ${\rm Ext}_\theta(\sfC)$ in direction $e_\theta$. In this example, ${\rm Ext}_\theta(\sfC)$ does not belong to $\sfC$ itself. The cell associated to the extremum is shaded in light blue and the line $\langle \cdot, e_\theta \rangle = {\rm E}_\theta(\sfC)$ is drawn in black.}
		\label{fig: extremal cell}
	\end{figure}
	
	Write $x_n$ for the primal vertex of $\bbL_{\rm mix}$ closest to $-ne_\theta$ and let $\sfC_{x_n}$ denote the cluster containing $x_n$. We define the IIC with three arms in the half-plane as the limiting measure
	\begin{align}
		\phi^{{\rm IIC},\theta}_{\bbL_{\rm mix},4}\left[\, \cdot \, \right] \coloneqq \lim_{n \to \infty} \phi_{\bbL_{\rm mix},4}\left[\, \cdot \,|\, {\rm Ext}_\theta(\sfC_{x_n}) = 0\right],
	\end{align}
	where the limit uses the weak convergence with respect to the product topology. In other words, the probability of any local event converges. 
	Under $\phi^{{\rm IIC},\theta}_{\bbL_{\rm mix},4}$, there exists an infinite cluster which is the limit of the clusters $\sfC_{x_n}$. We call this the incipient infinite cluster; 
	its extremal coordinate in the direction $e_\theta$ is $0$. 
	
	The measure $\phi^{{\rm IIC},\theta}_{\bbL_{\rm mix},4}$ describes the local environment around extrema of large clusters. Indeed, it may be shown that the neighbourhood of an extremum of a typical large cluster is distributed according to $\phi^{{\rm IIC},\theta}_{\bbL_{\rm mix},4}$, even when the cluster is conditioned on (reasonable) large scale features. The existence of the limit is a manifestation of the following mixing property, which in turn is a consequence of~\eqref{eq:RSW}.
	
	\begin{prop}\label{prop:IIC mixing at q=4}
		There exist constants $C, c_{\rm IIC} > 0$ such that the following holds for all $r \leq R/2$. For any configurations $\om_0$ on $\La_R^c$, any $x \in \La_R^c$ and any event $A$ depending only on edges in $\La_r$,
		\begin{align}
			\abs{\phi^{{\rm IIC},\theta}_{\bbL_{\rm mix},4}\left[A\right] - \phi_{\bbL_{\rm mix},4}\left[A \,|\, \om = \om_0 \text{ on } \La_R^c, {\rm Ext}_\theta(\sfC_x) = 0\right]} \leq C \big( \tfrac{r}{R}\big)^{c_{\rm IIC}},
		\end{align}
		as long as the conditioning is non-degenerate\footnote{The conditioning on $\{\om = \om_0 \text{ on } \La_R^c\}$ is always degenerate, but should be understood as imposing certain boundary conditions on the restriction of the measure to $\La_R$. Here, by non-degenerate, we mean that there exists at least one configuration in $\La_R$ such that $\{{\rm Ext}_\theta(\sfC_x)=0\}$ is realised.}.
	\end{prop}
	
	A model-specific construction of the IIC measure can be found in \cite{DCKozKraManOulRotationalInvarianceCritPlanar, OulThesis}; for its associated polynomial mixing rate, we refer to \cite[Prop.\ 3.1]{GarPetSchPivotalMeasure}. 
	
	\paragraph{Track-exchanges and drift.} 
	A crucial step in the proof of the asymptotic rotational invariance for $1 \leq q \leq 4$ of \cite{DCKozKraManOulRotationalInvarianceCritPlanar} is to analyse how the extremum of an IIC cluster is affected by track-exchanges.
	
	Let $\om^{\rm IIC}$ be sampled from $\phi^{\rm IIC,\theta}_{\bbL_{\rm mix},4}$ and write $\sfC^{\rm IIC}$ for the incipient infinite cluster of $\om^{\rm IIC}$. Define $\bfS_{\rm even}$ as the transformation obtained by simultaneously applying the track-exchanges $\bfT_{2k}$ for all $k \in \bbZ$, and $\bfS_{\rm odd}$ as the transformation obtained by applying the track exchanges $\bfT_{2k-1}$ for $k \in \bbZ$. Since the track-exchanges appearing in the transformations act on disjoint tracks and since the infinite-volume measures on $\bbL_{\rm mix}$ and its transforms are unique, one may perform these exchanges simultaneously on $\om^{\rm IIC}$.

	When $\bfS_{\rm even}$ is applied to $\bbL_{\rm mix}$, each track of angle $\alpha$ is exchanged with the track of angle $\beta$ directly above it. In particular, the vertices at the bottom of the $\alpha$-tracks remain fixed. Hence, any cluster containing one of these vertices admits a well-defined image under $\bfS_{\rm even}$. 
	Analogously, when $\bfS_{\rm odd}$ is applied to $\bfS_{\rm even}(\bbL_{\rm mix})$, the tracks of angle $\alpha$ are again exchanged with the tracks of angle $\beta$ directly above them, and clusters containing at least two edges thus have a natural image after the transformation.
	Note that $(\bfS_{\rm odd} \circ \bfS_{\rm even})(\bbL_{\rm mix})$ is a translate of $\bbL_{\rm mix}$ with the same cell structure. 
	
	Due to large clusters being preserved, $\sfC^{\rm IIC}$ has a corresponding cluster in $(\bfS_{\rm odd} \circ \bfS_{\rm even})(\om^{\rm IIC})$, 
	which we denote by $\tilde \sfC^{\rm IIC}$. We then define the \emph{IIC increment} by
	\begin{align}
		\Delta^{\rm IIC}{\rm E}_\theta = {\rm E}_\theta(\tilde \sfC^{\rm IIC})-{\rm E}_\theta(\sfC^{\rm IIC}).
	\end{align}
	
	Note that $\bfS_{\rm even}(\bbL_{\rm mix})$ is also a translate of $\bbL_{\rm mix}$. As such, one might expect the application of $\bfS_{\rm even}$ and the subsequent application of $\bfS_{\rm odd}$ to have identical effects. This is not the case: $\bfS_{\rm even}(\bbL_{\rm mix})$ and $\bbL_{\rm mix}$ are translates of each other, but their partition into cells differ, which affects the definition of ${\rm Ext}_\theta$.
	This is the reason for applying both transformations before considering the increment.
	
	The expected increment $\bbE[\Delta^{\rm IIC}{\rm E}_\theta]$ --- where $\bbE$ refers to the expectation taken with respect to the coupling between $\om^{\rm IIC}$ and $(\bfS_{\rm odd} \circ \bfS_{\rm even})(\om^{\rm IIC})$ --- captures the \emph{drift} of the extremum of a large cluster under the transformation $(\bfS_{\rm odd} \circ \bfS_{\rm even})$. It was proved in \cite{DCKozKraManOulRotationalInvarianceCritPlanar} that this drift vanishes, i.e., for all $\theta \in [0,2\pi)$,
		\begin{align}\label{eq:IIC drift is null}
			\bbE[\Delta^{\rm IIC}{\rm E}_\theta] = 0.
		\end{align}
	This fact plays a central role in establishing universality among isoradial rectangular lattices for~$1 \leq q \leq 4$. 
	
	\section{Half-plane measures}\label{sec: hp-measures}
	
	The strategy of proving universality by successively transforming one lattice into another via a sequence of track-exchanges is faced with additional difficulties for $q > 4$. In this regime, the boundary conditions influence the model at infinite distance. Indeed, suppose a lattice $\bbL$ is transformed into $\bbL'$ by a sequence of track-exchanges, which we denote by $\bfS$. The transformation affects the boundary conditions in an uncontrollable manner, which, in this case, are crucial. In particular, it is unclear whether the push-forward of $\phi^0_{\bbL,q}$ by $\bfS$ is $\phi^0_{\bbL',q}$. This problem does not appear in the regime $1 \leq q \leq 4$ due to the uniqueness of the infinite-volume measure. 
	To circumvent this issue, we will work with half-plane measures with free boundary conditions. 
	
	From here onwards, we consider lattices $\bbL(\bm\alpha)$ where $\bm\alpha = (\alpha_i)_{i\geq 1}$ is a half-infinite sequence of angles; as in the previous section, these sequences will contain at most two distinct values. 
	These lattices cover the upper half-plane $\bbR \times  \bbR_{\geq0}$. 
	For such a lattice, write $\partial \bbL(\bm\alpha)$ for the set of its vertices lying on the horizontal axis $\bbR \times \{0\}$.
	
	\subsection{Uniqueness of the half-plane free measure} 
	
	We call a half-infinite sequence of angles  $\bm\alpha = (\alpha_i)_{i\geq 1}$ \emph{periodic} if there exists $k \in \bbN$ such that $\alpha_{k+i} = \alpha_i$ for all $i\geq 1$.
	
	\begin{prop}\label{prop:hp_measure}
		Fix $q \geq 1$ and a half-infinite periodic sequence of angles $\bm\alpha$. 
		For any sequence $\xi_n$ of boundary conditions on $\bbL(\bm\alpha)\cap \La_n$ with the property that $\xi_n$ is free on $\partial \bbL(\bm\alpha)$, 
		the measures $\phi_{\bbL(\bm\alpha)\cap \La_n,q}^{\xi_n}$ converge  as $n \to \infty$  to a measure $\phi_{\bbL(\bm\alpha),q}^0$, 
		which we call the \emph{free half-plane measure}.
	\end{prop}
	
	\begin{proof}
		Write $\phi^{1/0}_{\bbL(\bm\alpha) \cap \La_n,q}$ for the isoradial random-cluster measure on $\bbL(\bm\alpha) \cap \La_n$ with free boundary conditions on $[-n,n] \times \{0\}$ and wired boundary conditions for the rest of the boundary. Let $\phi^{1/0}_{\bbL(\bm\alpha),q}$ be the half-plane random-cluster measure which is the weak (decreasing) limit of $\phi^{1/0}_{\bbL(\bm\alpha) \cap \La_n,q}$ for $n \to \infty$.
		
		Similarly, write $\phi^0_{\bbL(\bm\alpha) \cap \La_n,q}$ for the random-cluster measure on $\bbL(\bm \alpha) \cap \La_n$ with free boundary conditions everywhere and denote its weak (increasing) limit by $\phi^0_{\bbL(\bm\alpha),q}$.
		 
		By \eqref{eq:CBC}, it suffices to show that 
		\begin{align}
		\phi^{1/0}_{\bbL(\bm \alpha),q} = \phi^{0}_{\bbL(\bm \alpha),q},
		\end{align}
		which in turn follows from the fact that, under $\phi^{1/0}_{\bbL(\bm \alpha),q}$, there exists a.s.\ no infinite cluster. 
		The absence of an infinite cluster in $\phi^{1/0}_{\bbL(\bm \alpha),q}$ follows from the arguments of \cite{GeoHigGibbsIsing, GlaManGibbsFKPotts},
		with the periodicity of the sequence of angles playing a particular role.
		We give a brief sketch for completeness. 
		
		Consider the increasing limit $\phi$ of the downward translates of $\phi^{1/0}_{\bbL(\bm \alpha),q}$.
		Then $\phi$ is a percolation measure on a periodic isoradial lattice $\bbL$. Furthermore, $\phi$ itself is invariant under the translations which map $\bbL$ to itself.
		The dual of $\phi$ produces configurations on the horizontal translate of $\bbL$ by $1$. 
		Furthermore, due to the order in which the different parts of the boundary were taken to infinity, $\phi$ is stochastically dominated by its dual (translated by $(-1,0)$). 
		
		Due to the planarity and periodicity of $\bbL$, the non-coexistence theorem of \cite{SheRandomSurfaces,DumRaoTasSharpnessFK} applies, and we conclude that, under $\phi$, either the primal or the dual configuration contains no infinite cluster.
		By the domination above, the primal configuration contains a.s.\ no infinite cluster, which extends to $\phi^{1/0}_{\bbL(\bm \alpha),q}$ by stochastic domination. 
	\end{proof}

	\begin{cor}[Track-exchange for half-plane measures]\label{cor:track_exchange_hp}
		Fix $q \geq 1$. Let $\bm\alpha$ be a half-infinite periodic sequence of angles and $A \subset \bbN$ be a periodic set containing no consecutive integers. 
		Write $\bm\alpha'$ for the half-infinite periodic sequence obtained from $\bm\alpha$ by exchanging $\alpha_{i-1}$ and $\alpha_{i}$ for each $i\in A$.
		If $\omega$ is sampled according to $\phi_{\bbL(\bm\alpha),q}^0$ and $\omega'$ is the configuration obtained by applying each $(\bfT_i)_{i\in A}$ to $\omega$, 
		then $\omega'$ is distributed according to $\phi_{\bbL(\bm\alpha'),q}^0$.
	\end{cor}

	\begin{proof}
		Let $\bbS_K$ be the strip of $\bbL(\bm\alpha)$ formed of the horizontal tracks $t_1,\dots, t_K$ and similarly
		let $\bbS_K'$ be the strip of $\bbL(\bm\alpha')$ containing the tracks $t_1,\dots, t_K$. 
		Write $\phi^0_{\bbS_K,q}$ and $\phi^0_{\bbS'_K,q}$ for the measures on $\bbS_K$ and $\bbS'_K$, respectively, with free boundary conditions.

		Write $\bfS$ for the composition of all transformations $(\bfT_i)_{i\in A}$.
		Assume now that $K$ is such that $K+1\notin A$. Then, we may apply all track exchanges $(\bfT_i)_{i\in A; i\leq K}$ to $\bbS_K$ to obtain $\bbS'_K$. 
		For simplicity, we call the composition of these transformations also $\bfS$, since the transformations $(\bfT_i)_{i\in A; i> K}$ do not apply to $\bbS_K$ and hence are considered trivial in this setting. 
	
		If $\omega$ is sampled according to $\phi^0_{\bbS_K,q}$, then, by Proposition~\ref{prop:track-exchange operator},  the law of $\bfS(\omega)$ is given by $\phi^0_{\bbS'_K,q}$.
		Proposition~\ref{prop:hp_measure} then implies that
		\begin{align}
			\phi_{\bbL(\bm\alpha),q}^0 = \lim_{K\to \infty} \phi^0_{\bbS_K,q} \quad\text{ and } \quad 
			\phi_{\bbL(\bm\alpha'),q}^0 = \lim_{K\to \infty} \phi^0_{\bbS'_K,q}.
		\end{align} 
		The conclusion follows. 
	\end{proof}

	\begin{cor}\label{cor:top_does_not_matter}
		Fix $q \geq 1$ and let $\bm\alpha$ and $\bm\beta$ be two half-infinite periodic sequences of angles with $\alpha_i = \beta_i$ for all $i\leq n$. 
		Then, for any event $H$ depending only on the edges on the tracks $t_1, \dots, t_n$,
		\begin{align}
			\phi_{\bbL(\bm\alpha),q}^0[H] = \phi_{\bbL(\bm\beta),q}^0[H].
		\end{align}
	\end{cor}

	\begin{proof}
		Let $\bbS(\theta_1,\dots,\theta_K)$ denote the strip with $K$ horizontal tracks of transverse angles $\theta_1,\dots,\theta_K$ and vertical tracks of transverse angle $0$. 
		Let $\phi^\xi_{\bbS(\theta_1,\dots,\theta_K),q}$ denote the measure on this strip with boundary conditions $\xi$. 

		Fix $H$ as in the statement and $\eps > 0$. Then, by Proposition~\ref{prop:hp_measure} applied to both $\bbL(\bm\alpha)$ and $ \bbL(\bm\beta)$, there exists $K \geq n$ such that
		\begin{align}\label{eq:HHLL}
			\big|\phi_{\bbL(\bm\alpha),q}^0[H] - \phi_{\bbS(\alpha_1,\dots,\alpha_K),q}^\xi[H]\big| \leq \eps \quad \text{ and }\quad
			\big|\phi_{\bbL(\bm\beta),q}^0[H] - \phi_{\bbS(\beta_1,\dots,\beta_K),q}^\xi[H]\big| \leq \eps 
		\end{align}
		for any boundary conditions $\xi$ that are free on the bottom of the strip. 
	
		Consider now the strips
		\begin{align}
			\bbS \coloneqq \bbS(\alpha_1,\dots,\alpha_K, \beta_{n+1}, \dots, \beta_K)   \quad \text{ and }\quad 	\bbS' \coloneqq \bbS(\beta_1,\dots,\beta_K, \alpha_{n+1}, \dots, \alpha_K).
		\end{align}
		Due to \eqref{eq:HHLL} and \eqref{eq:SMP}, 
		\begin{align}\label{eq:HHLL1}
			\big|\phi_{\bbL(\bm\alpha),q}^0[H] - \phi_{\bbS,q}^0[H]\big| \leq \eps \quad \text{ and }\quad
			\big|\phi_{\bbL(\bm\beta),q}^0[H] - \phi_{\bbS',q}^0[H]\big| \leq \eps.
		\end{align}
	
		Since $\alpha_i = \beta_i$ for all $i\leq n$, there exists a sequence of track exchanges that act only above the track $t_n$ and turn $\bbS$ into $\bbS'$. 
		The configuration on $t_1,\dots,t_n$ does not change when applying this sequence of transformations, so we find 
		\begin{align}\label{eq:HHLL2}
			 \phi_{\bbS,q}^0[H] = \phi_{\bbS',q}^0[H].
		\end{align}
		Combining this with \eqref{eq:HHLL1}, we conclude that $|\phi_{\bbL(\bm\alpha),q}^0[H] - \phi_{\bbL(\bm\beta),q}^0[H]| \leq 2\eps$.
		Finally, since $\eps > 0$ is arbitrary, we find $\phi_{\bbL(\bm\alpha),q}^0[H] =\phi_{\bbL(\bm\beta),q}^0[H]$.
	\end{proof}
	
	\subsection{Rate of decay of connection probabilities in the half-plane}

	For $\alpha \in (0,\pi)$, write $\bbL_+(\alpha)$  for the half-plane isoradial lattice with constant transverse angle $\alpha$ for the tracks $(t_{n})_{n\geq1}$. 
	Since we will exclusively be working with half-plane measures, we now introduce the half-plane analogue of $\zeta$.
	For $\theta \in [0,2\pi)$, set
	\begin{align}
		\zeta_{\alpha,q}^{\rm hp}(\theta) = \Big(\lim_{n\to \infty} - \tfrac1n \log \phi^0_{\bbL_+(\alpha),q}\big[0 \lr \calH_{\geq n}^{\theta}\big]\Big)^{-1}.
	\end{align}
	The existence of the limit, similarly to that of $\zeta_{\alpha,q}(\theta)$, is proved using the super-additivity of the sequence
	\begin{align}
		\max_{x \in  \calH_{\geq n}^{\theta}}  \log \phi^0_{\bbL_+(\alpha),q}[0 \lr x].
	\end{align}
	Finally, observe that, due to \eqref{eq:exponential_decay_L(alpha)} and the finite energy property \eqref{eq:FE}, 
	\begin{align}
 		0<\zeta_{\alpha,q}^{\rm hp}(\theta) < \infty
	\end{align}
	for any $q > 4$, $\alpha \in (0,\pi)$, and $\theta \neq 3\pi/2$. 

	We call a direction $\theta$ {\em upper-half-plane aiming} for $\bbL(\alpha)$ if there exists an infinite number of integers $n \geq 1$ such that 
	the point $x \in   \calH_{\geq n}^{\theta}$ maximising $\phi^0_{\bbL(\alpha),q}[0 \lr x\,]$ lies in the upper half-plane. 
	The upper-half-plane aiming directions may be shown to be those dual --- in the sense of \eqref{eq:convex-duality_iso} --- to $\theta \in [0,\pi]$.
	For $\bbL(\pi/2)$, all $\theta \in [0,\pi]$ are upper-half-plane aiming due to the invariance of the model with respect to vertical reflections.
	This is not necessarily the case for $\bbL(\alpha)$, as this particular symmetry may be lost. 

	We note two fundamental facts that are due to the invariance of the models $\phi^0_{\bbL(\alpha),q}$ under rotation by $\pi$. For any $\theta \in [0,\pi)$ and $\alpha  \in (0,\pi)$,
	\begin{align}
		&\zeta_{\alpha,q}(\theta) = 	\zeta_{\alpha,q}(\pi+ \theta) \quad \text{ and } \quad 
		\text{at least one of $\theta$ and $\pi+\theta$ is upper-half-plane aiming}.
		\label{eq:upper-half-plane-aiming}
	\end{align}
	Finally, we state the essential link between $\zeta_{\alpha,q}^{\rm hp}(\theta)$ and $\zeta_{\alpha,q}(\theta)$.

	\begin{prop}\label{prop:hp_cor_len}
		Fix $\alpha \in (0,\pi)$, $q > 4$, and $\theta \in [0,2\pi)$. Then 
		\begin{align}\label{eq:zeta_zeta_hp}
			\zeta_{\alpha,q}(\theta) \geq \zeta_{\alpha,q}^{\rm hp}(\theta),
		\end{align}
		with equality if $\theta$ is upper-half-plane aiming for $\bbL(\alpha)$.
	\end{prop}
	
	\begin{proof}
		By \eqref{eq:CBC} and inclusion of events, 
		\begin{align}
			\phi^0_{\bbL_+( \alpha),q}\big[0 \lr\calH_{\geq n}^{\theta}\big] \leq \phi^0_{\bbL( \alpha),q}\big[0 \lr \calH_{\geq n}^{\theta}\big].
		\end{align}
		Taking the logarithm, dividing by $-n$ and taking the limit $n \to \infty$, we obtain \eqref{eq:zeta_zeta_hp}. 
	
		Assume now that $\theta$ is upper-half-plane aiming for $\bbL(\alpha)$.
		Fix $\eps >0$. Then, for $r$ larger than some threshold,
		\begin{align}
			 \phi^0_{\bbL( \alpha),q}\big[0 \lr \calH_{\geq r}^{\theta}\big] \geq   \exp\big(-\tfrac{r}{\zeta_{\alpha,q}(\theta) - \eps}\big).
		\end{align}
		
		Consider the point $x \in   \calH_{\geq r}^{\theta}$ that maximises $\phi^0_{\bbL(\alpha),q}[0 \lr x \,]$. 
		In light of \eqref{eq:exponential_decay_L(alpha)}, there exists a constant $C >0$ independent of $r$ such that
		\begin{align}
			 \phi^0_{\bbL( \alpha),q}[0 \lr x] \geq   \frac{1}{C r} \exp\big(-\tfrac{r}{\zeta_{\alpha,q}(\theta) - \eps}\big).
		\end{align}
		By further increasing $r$ and using the fact that $\theta$ is upper-half-plane aiming, we may consider $x$ to be in the upper-half-plane and such that 
		\begin{align}
			 \phi^0_{\bbL( \alpha),q}[0 \lr x] \geq   \exp\big(-\tfrac{r}{\zeta_{\alpha,q}(\theta) - 2\eps} \big).
		\end{align}
		Finally, we may find $R\geq 1$ large enough such that 
		\begin{align}\label{eq:0tox_zeta}
			\phi^0_{\La_R \cap \bbL( \alpha),q}[0 \lr x] \geq   \exp\big(-\tfrac{r}{\zeta_{\alpha,q}(\theta) - 3\eps}\big).
		\end{align}
				
		Now fix $z \in \bbL_+(\alpha)$ at a distance at least $R$ from the horizontal line.
		Then, for any $n$ sufficiently large, by \eqref{eq:FKG} and \eqref{eq:CBC}
		\begin{align}
			\phi^0_{\bbL_+(\alpha),q}\big[0 \lr \calH_{\geq n}^{\theta}\big] 
			&\geq \phi^0_{\bbL_+(\alpha),q}[0 \lr z] 	\phi^0_{\bbL_+(\alpha),q}[z \lr z + \ell x ]\\
			&\geq \phi^0_{\bbL_+(\alpha),q}[0 \lr z ] 	\phi^0_{\La_R \cap \bbL( \alpha),q}[0 \lr x]^{\ell}
		\end{align}
		when $\ell$ is a sufficiently large integer such that $z + \ell x  \in \calH_{\geq n}^{\theta}$. Due to the choice of $x$, 
		one may choose $\ell \leq \frac{n + C_0}{r}$ for some constant $C_0$ that depends on $\theta, r, x$, and $z$, but not on $n$. 
		Using \eqref{eq:0tox_zeta}, we conclude the existence of a constant $c_1 > 0$ independent of $n$ such that 
		\begin{align}
			\phi^0_{\bbL_+(\alpha),q}\big[0 \lr \calH_{\geq n}^{\theta}\big] 
			\geq c_1 \exp\big(-\tfrac{n}{\zeta_{\alpha,q}(\theta) - 3\eps}\big).
		\end{align}
		Taking the logarithm, dividing by $-n$ and taking $n$ to infinity, we conclude that 
		\begin{align}
			\zeta_{\alpha,q}^{\rm hp}(\theta) \geq  \zeta_{\alpha,q}(\theta)  - 3\eps.
		\end{align}
		Since $\eps > 0$ was chosen arbitrarily and in light of \eqref{eq:zeta_zeta_hp}, we conclude that the two quantities are equal. 
	\end{proof}

	\section{Universality: proof of Theorem~\ref{thm:universality}}\label{sec:universality}
	
	We briefly outline the idea behind proving Theorem~\ref{thm:universality}. Our aim will be to compare 
	$\zeta_{\alpha,q}^{\rm hp}(\theta)$ for different values of $\alpha$ and $\theta \neq 3\pi/2$ fixed. More precisely, we will show the following. 
	
	\begin{prop}\label{prop:universality}
		For all $\eps > 0$, there exists $q_0 > 4$ such that for $q \in (4, q_0]$, all $\alpha,\beta \in (\eps,\pi-\eps)$ and any $\theta \in [0,2\pi)$ with $\theta \neq 3\pi/2$,
		\begin{align}
			\bigg|\frac{\zeta_{\alpha,q}^{\rm hp}(\theta)}{\zeta_{\beta,q}^{\rm hp}(\theta)}-1\bigg| < \eps.
		\end{align}
	\end{prop}
	
	We will see in Section~\ref{sec:conclusion_universality} how to pass from the half-plane connection rates to the full-plane connection rates, and hence prove Theorem~\ref{thm:universality}.
	The rest of the section is dedicated to proving Proposition~\ref{prop:universality}.
	
	The aim is to show that connection probabilities in $\om \sim \phi^0_{\bbL_+(\alpha),q}$ and $\om' \sim \phi^0_{\bbL_+(\beta),q}$ are close to each other when $q > 4$ is taken sufficiently close to $4$. 
	To achieve this, we couple these two configurations and construct a sequence of intermediate configurations, each obtained from the previous one by a sequence of track-exchanges.
	
	In this coupling, we keep track of the extremal coordinate ${\rm E}_\theta(\sfC)$ of the cluster $\sfC$ of the origin in direction $e_\theta$. 
	By taking $q > 4$ close enough to $4$, we argue that the effect of track-exchanges on ${\rm E}_\theta(\sfC)$ are almost identical in law to $\Delta^{\rm IIC}{\rm E}_\theta$, 
	and therefore have vanishingly small expectation. 
	This will allow us to compare the probabilities in the initial and final configurations $\omega$ and $\omega'$ to have ${\rm E}_\theta(\sfC) \geq n$ for large values of $n$. 
	Ultimately, our goal is to show that the exponential rate of decay of these two quantities is almost equal.

	\subsection{The coupling}
	
	We will now fix some notation necessary for the proof of Proposition~\ref{prop:universality} and describe the coupling in detail. 
	Fix two angles $\alpha, \beta \in (0,\pi)$ and $N \geq 1$ even. 
	All constants below will be uniform in $\alpha$ and $\beta$ in compacts of $(0,\pi)$ and in $N$.  
	
	Write $\bbL_0$ for the lattice $\bbL_+(\bm \alpha)$, where $\bm \alpha$ is the sequence given by 
	\begin{align}
		\alpha_i = 
		\begin{cases}
			\alpha & \text{ for $2kN < i \leq (2k+1)N$, $k \geq 0$,}\\
			\beta &\text{ otherwise.} 
		\end{cases}
	\end{align}
	We can partition the lattice into blocks of $N$ tracks of constant angle. By means of the track-exchange operator, we can define a sequence of lattices $(\bbL_t)_{t \geq 0}$ where the blocks of angle $\alpha$ are eventually exchanged with those of angle $\beta$.
	
	Recall that we write $\bfT_i$ for the track-exchange operator exchanging the tracks $t_i$ and $t_{i-1}$. If $t_i$ and $t_{i-1}$ have the same transverse angle, we set $\bfT_i = {\rm id}$. 
	For $t \geq 1$, apply the following sequence of track exchanges to $\bbL_{t}$ in order to obtain $\bbL_{t+1}$:
	\begin{align}
		\bfS_t = 
		\begin{cases}
		\bfT_3 \circ  \bfT_5 \circ  \dots \bfT_{2N-1} \circ \bfT_{2N + 3} \circ \bfT_{2N + 5} \circ\dots    &\quad \text{if $t$ is odd,}\\
		 \bfT_2 \circ \bfT_4 \circ \dots  &\quad \text{if $t$ is even.}
		\end{cases}
	\end{align}
	That is, for $t \geq 1$, set $\bbL_{t} = \bfS_{t-1} \circ \dots\circ \bfS_0(\bbL_0)$. See Figure~\ref{fig: block-exchange} for an illustration; note that since $N$ is even, $\bfS_0$ is trivial and $\bbL_1 = \bbL_0$.

	The transformations $\bfS_{t}$ for odd times $t$ contains all track-exchanges $\bfT_i$ with $i$ odd, {\em except} for $i -1 \in  2N\bbZ$. 
	This ensures that there is never any exchange of tracks between blocks of $2N$ successive tracks, and ultimately ensures that all lattices $\bbL_t$ are $2N$-periodic.
	
	Each period is formed of a \emph{mixed block} of alternating tracks of angles $\alpha$ and $\beta$ sandwiched between an $\alpha$-block and a $\beta$-block; the sizes of the blocks depend on $t$ and may be null. Lines separating different blocks of the lattice (including the horizontal axis) will be referred to as \emph{interfaces}.
	
	After $2N$ transformations, the mixed block disappears and the $\beta$-block and $\alpha$-block have been exchanged completely. We refer again to Figure~\ref{fig: block-exchange}.

	\begin{figure}
		\centering
		\footnotesize
		
	\def\svgwidth{1\columnwidth}
\begingroup%
  \makeatletter%
  \providecommand\color[2][]{%
    \errmessage{(Inkscape) Color is used for the text in Inkscape, but the package 'color.sty' is not loaded}%
    \renewcommand\color[2][]{}%
  }%
  \providecommand\transparent[1]{%
    \errmessage{(Inkscape) Transparency is used (non-zero) for the text in Inkscape, but the package 'transparent.sty' is not loaded}%
    \renewcommand\transparent[1]{}%
  }%
  \providecommand\rotatebox[2]{#2}%
  \newcommand*\fsize{\dimexpr\f@size pt\relax}%
  \newcommand*\lineheight[1]{\fontsize{\fsize}{#1\fsize}\selectfont}%
  \ifx\svgwidth\undefined%
    \setlength{\unitlength}{505.74921375bp}%
    \ifx\svgscale\undefined%
      \relax%
    \else%
      \setlength{\unitlength}{\unitlength * \real{\svgscale}}%
    \fi%
  \else%
    \setlength{\unitlength}{\svgwidth}%
  \fi%
  \global\let\svgwidth\undefined%
  \global\let\svgscale\undefined%
  \makeatother%
  \begin{picture}(1,0.23032333)%
    \lineheight{1}%
    \setlength\tabcolsep{0pt}%
    \put(0,0){\includegraphics[width=\unitlength,page=1]{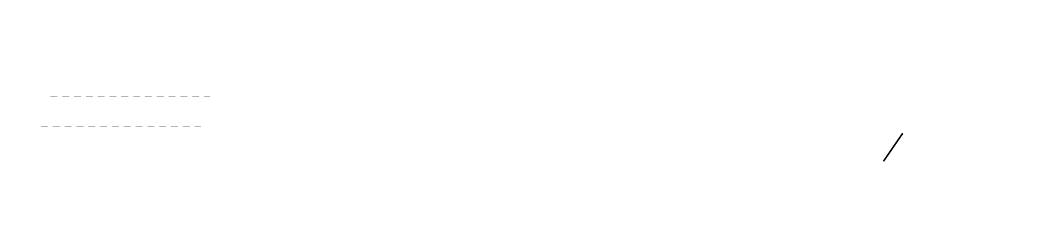}}%
    \put(0.00055645,0.13369075){\color[rgb]{0,0,0}\makebox(0,0)[lt]{\lineheight{1.25}\smash{\begin{tabular}[t]{l}$t_{N+1}$\end{tabular}}}}%
    \put(0.01366667,0.10696922){\color[rgb]{0,0,0}\makebox(0,0)[lt]{\lineheight{1.25}\smash{\begin{tabular}[t]{l}$t_N$\end{tabular}}}}%
    \put(0,0){\includegraphics[width=\unitlength,page=2]{block-exchange.pdf}}%
  \end{picture}%
\endgroup%

		\caption{The track-exchanges performed on a single period. From left to right: The initial lattice $\bbL_0$ with tracks of angle $\alpha = \pi/2$ at the bottom and $\beta < \pi/2$ at the top. The first two tracks to be exchanged are marked with dashed grey lines. Applying $\bfS_t \circ \dots \circ \bfS_0$ transforms $\bbL_0$ into $\bbL_{t+1}$ and a mixed block (colored in red) starts appearing in the middle. The cells in the mixed block at even timesteps are centred around rhombi with angle $\beta$. After more transformations, a $\beta$-block starts forming at the bottom and an $\alpha$-block at the top (both marked in blue). By time $2N$, the $\beta$-block and the $\alpha$-block have been exchanged completely. The interfaces of each lattice are marked with thick red lines.}
		\label{fig: block-exchange}
	\end{figure}
	
	Fix $q > 4$. We will associate a sequence of configurations $(\om_t)_{t \geq 0}$ to the lattices $(\bbL_t)_{t \geq 0}$ with the property that $\om_t \sim \phi^0_{\bbL_t,q}$ for all $t \geq 0$. This produces a \emph{coupling} of the measures $(\phi^0_{\bbL_t,q})$. 
	The coupling is constructed as follows. First, sample $\om_0$ according to $\phi^0_{\bbL_0,q}$. Then, for $t \geq 0$, assuming that $\om_0, \dots, \om_{t}$ are already defined, set
	\begin{align}
		\om_{t+1} = \bfS_{t}(\om_{t}).
	\end{align}
	Corollary~\ref{cor:track_exchange_hp} ensures that $\om_{t}$ indeed has law $\phi^0_{\bbL_{t},q}$ for all $t\geq 1$.
	We write $\bbP$ for the probability measure governing the random sequence $(\om_t)_{t \geq 0}$.
	
	Fix an angle $\theta \in [0,2\pi) \setminus \{3\pi/2\}$. For each configuration $\om_t$, write $\sfC_t$ for the cluster of the origin in $\om_t$ and recall that ${\rm Ext}_\theta(\sfC_t)$ (resp.\ ${\rm E}_\theta(\sfC_t)$) denotes the extremum (resp.\ extremal coordinate) of $\sfC_t$ in direction $e_\theta$. Since $\theta$ remains fixed throughout, we will henceforth omit it from the notation.

	We will be interested in the evolution of $({\rm E}(\sfC_t))_{t \geq 0}$. 
	In order to keep track of the dynamics of the process, we define, for $t \geq 0$ even, the increments
	\begin{align}
		\Delta_t {\rm E} = {\rm E}(\sfC_{t+2}) - {\rm E}(\sfC_t).
	\end{align}
	
	The following estimate, bounding the expected increment, conditionally on the extremal coordinate being large, is a key step in the proof of Proposition~\ref{prop:universality}.
	
	\begin{prop}\label{prop:conditional_drift}
		For each $t\geq 0$ even, $\Delta_t {\rm E}$ is a.s.\ bounded by 4. 
		Moreover, for any $\delta,\eps > 0$, there exists $q_0 > 4$ such that for any $q \in (4,q_0]$, $n, N$ sufficiently large and any $1 \leq t \leq N(1-\delta)$ or $(1+\delta) N \leq t \leq 2N$, it holds that
		\begin{align}\label{eq:conditional_drift}
			\bbE \big[\Delta_t {\rm E}\,|\, {\rm E}(\sfC_t) \in [n\eps,(n+1)\eps)\big] \geq -\delta.
		\end{align}
		Moreover, $q_0$ may be chosen independently of $\theta$,  $n$ and $N$ (both sufficiently large) and uniformly in $\alpha,\beta$ in compacts of $(0,\pi)$.
	\end{prop}
	
	Note that we do not follow the precise evolution of ${\rm E}(\sfC_t)$, but rather a ``rounding'' of ${\rm E}(\sfC_t)$ at an arbitrary precision. 
	Indeed, the exact value of  ${\rm E}(\sfC_t)$ may dictate the exact position of ${\rm Ext}(\sfC_t)$, which may lead to a degenerate conditioning. 
	See the proof of Lemma~\ref{lem:extremum_close_to_interface} for the use of this rounding. 
	
	We excluded $t$ between $(1-\delta)N$ and $(1+\delta)N$ from \eqref{eq:conditional_drift} for convenience. Indeed, in these time steps, the mixed block comes close to the horizontal axis, which entails some additional complications. While we believe~\eqref{eq:conditional_drift}  to also apply in this case, we omit it as it is not strictly necessary for the rest of the proof.

	As a consequence of Proposition~\ref{prop:conditional_drift}, 
	we establish that the coupling is indeed such that the corresponding point-to-hyperplane connection probabilities in $\bbL_0$ and $\bbL_{2N}$ 
	differ only negligibly if $q$ is chosen sufficiently close to $4$ and the distance $n$ to the hyperplane is large enough.
	
	\begin{cor}\label{cor: extremum barely moves}
		For any $\eta > 0$ and $C \geq 1$, there exists $q_0 > 4$ such that for $q \in (4, q_0]$ and all $n$ sufficiently large, if we set $N = Cn$, we have 
		\begin{align}\label{eq:extremum barely moves}
			\bbP\big[{\rm E}(\sfC_{2N}) \geq (1-\eta)n\big] \geq \eta \bbP\big[{\rm E}(\sfC_0) \geq n\big].
		\end{align}
		Moreover, $q_0$ may be chosen independently of $\theta$ and uniformly in $\alpha$ and $\beta$ in compacts of $(0,\pi)$. 
	\end{cor}
	
	\begin{proof}
		Fix $\eta > 0$ and $C \geq 1$. With no loss of generality, we may assume $\eta < 1/2C$. 
		Choose constants $\eps,\delta > 0$ so that
		\begin{align}
			5\delta + \eps < \tfrac{\eta}{2C} \quad \text{and} \quad \eta \leq \tfrac{5 \delta + \eps}{2(1 + 5 \delta + \eps)}.
		\end{align}
		Then choose $q_0 > 4$, sufficiently close to $4$ so that Proposition~\ref{prop:conditional_drift} holds for the chosen values of $\delta, \eps$. 
		Fix now $n$ sufficiently large that~\eqref{eq:conditional_drift} applies for all values above $\eps m\geq n/2$ and set $N = Cn$. 
		
		Define the $\eps$-rounding $e_t = \lfloor {\rm E}(\sfC_t)/\eps \rfloor \cdot \eps$ of ${\rm E}(\sfC_t)$ 
		and write $\Delta_t e = e_{t+2}-e_t$ for all $t \geq 0$ even. Then, Proposition~\ref{prop:conditional_drift} states that, for any $0 \leq t \leq 2N$ even, except if $(1-\delta)N \leq t \leq (1+\delta)N$,
		\begin{align}\label{eq:e_increment}
			\bbE \big[\Delta_t e \,\big|\, e_t = \eps m \big] \geq -\delta -\eps
		\end{align}
		for all $q \in (4,q_0]$ and $\eps m \geq n/2$.

		Note that \eqref{eq:e_increment} does not provide a lower bound on the expected increment $\Delta_t e$ given the full past of the process.
		As our ultimate goal is to study the process under the conditioning $e_0 \geq n$, this may produce difficulties. 		
	 	To circumvent them, we will modify the process $(e_{2t})_{t \geq 0}$ to render it Markov, while maintaining the marginal laws. 
		This is a standard procedure: define a process $(\tilde e_{2t})_{t \geq 0}$ on a potentially extended probability space as follows. 
		\begin{itemize}
		\item 
		Let $\tilde e_0 = e_0$. 
		\item For any $t \geq 0$, let $\om_{t+1/2}$ be a configuration with the same law as $\om_t$, sampled independently of the past, but such that $\lfloor {\rm E}(\sfC_{t + 1/2})/\eps\rfloor \cdot \eps = \tilde e_t$. 
		 \item Define $\omega_{t+2} =  \bfS_{t+1} \circ \bfS_t(\om_{t+1/2})$ and use it to compute $\tilde e_{t+2}$. 
		\end{itemize} 
		This resampling procedure renders the process $(\tilde e_{2t})_{t \geq 0}$ Markov. 
		Additionally, it is easy to check that, for any fixed, even $t$, $e_t$ and $\tilde e_t$ have the same law. 
		Finally, \eqref{eq:e_increment} also applies to the process $(\tilde e_t)_{t \geq 0}$.
		
		Set $\tau = \inf\{ t \geq 0: \tilde e_{2t} < n/2\}$. 
		By \eqref{eq:e_increment}, we conclude that $\big( \tilde e_{2(t\wedge \tau)} + (\delta + \eps)t \big)_{0 \leq 2t\leq (1-\delta)N}$
		and $\big( \tilde e_{2(t\wedge \tau)} + (\delta + \eps)t \big)_{(1+\delta)N \leq 2t\leq 2N}$
		are submartingales with bounded increments. 
		Furthermore, $ \tilde e_{(1+\delta)N \wedge 2\tau} -  \tilde e_{(1-\delta)N \wedge 2\tau} \leq 4\delta N$, due to the deterministic bound on $\Delta_t \tilde e$. 
		
		Thus, by the optional stopping theorem,
		\begin{align}
			\bbE\big[\tilde e_{2(N\wedge \tau)} - \tilde e_{0} \,\big|\,\tilde e_0\big] \geq - (5\delta + \eps)N.
		\end{align}

		Again, by the deterministic bound on $\Delta_t \tilde e$, we obtain $\tilde e_{2(N\wedge \tau)} - \tilde e_{0} \leq 4N$. 
		Applying the Markov inequality, we conclude that 
		\begin{align}
			\bbP\big[\tilde e_{2(N\wedge \tau)} - \tilde e_{0} > -\eta n \,\big|\,\tilde e_0 \geq n \big] 
			&\geq \bbP\big[\tilde e_{2(N\wedge \tau)} - \tilde e_{0} > -2( 5 \delta + \eps) N \,\big|\,\tilde e_0 \geq n \big] \\
			&\geq \tfrac{ 5 \delta + \eps}{2(1 + 5 \delta + \eps)} 
			\geq \eta.
		\end{align}
		
		Since $\eta n  < n/2$, when the event in the first line occurs, $\tilde e_{2(N\wedge \tau)} > n/2$, and thus $\tau \geq N$. 
		We conclude from the above that 
		\begin{align}\label{eq:tildee_lb}
			\bbP\big[\tilde e_{2N} > (1-\eta) n \,\big|\,\tilde e_0 \geq n \big] 
			\geq \eta.
		\end{align}
		Finally, since $\tilde e$ and $e$ have the same marginals, we find that 
		\begin{align}
			\bbP\big[{\rm E}(\sfC_{2N}) \geq (1-\eta)n\big] 
			&\geq \bbP\big[\tilde e_{2N} \geq (1-\eta)n\big] \\
			&\geq \eta\bbP\big[ \tilde e_0 \geq n \big]
			 = \eta \bbP\big[{\rm E}(\sfC_0) \geq n\big].
		\end{align}
		In the last line, we assumed for simplicity that $n/\eps$ is integer-valued. This concludes the proof.

		We close by remarking that, by using the independence of the increments, it may be proved that 
		$\bbP[\tilde e_{2N} > (1-\eta) n \,|\,\tilde e_0 \geq n ]  \to 1$ as $n\to \infty$. 
		For our goal, the weaker bound \eqref{eq:tildee_lb} suffices. 
	\end{proof}
	
	\subsection{Expected increment: proof of Proposition~\ref{prop:conditional_drift}}
	
	Fix $\alpha$, $\beta$, $N$, $n$, $\theta$ and $t \geq 0$ even. 
	All constant below are independent of these choices, with $\alpha$ and $\beta$ being taken in a compact of $(0,\pi)$ and $n,N$ being large enough. 
	Also fix $\delta,\eps > 0$. 
	
	First, we argue that $\Delta_t {\rm E}$ is deterministically bounded by $4$.
	Split the vertices of $\bbL_t$ and $\bbL_{t+1}$ into those placed at the top of tracks $t_i$ with $i$ even
	and those placed at the bottom of such tracks --- this is a bi-partition of $\bbL_t$. 
	The former type of vertex is not affected by the transformation $\bfS_{t}$, and thus the vertices of this type in $\sfC_t$ are the same as those in $\sfC_{t+1}$. 
	Finally, for both $\omega_{t}$ and $\omega_{t+1}$, any vertex of the second category that is connected to $0$ is connected to a vertex of the first category by an open edge whose length is at most $2$. 
	Repeating the same argument when applying $\bfS_{t+1}$ to $\om_{t+1}$ (with the roles of the two types of vertices reversed), we conclude that $\sfC_{t+2}$ must contain a vertex in a cell neighbouring that of ${\rm Ext}(\sfC_{t})$, and vice versa.
	Given that the diameter of a cell is bounded by $4$, this provides the deterministic bound $\Delta_t {\rm E} \leq 4$.
	
	We now turn to \eqref{eq:conditional_drift}. Assume now that $1 \leq t \leq N(1-\delta)$ or $(1+\delta) N \leq t \leq 2N$. The expected increment can be decomposed as follows:
	\begin{align}
		\begin{aligned}\label{eq:Delta_sum}
		\bbE &\big[\Delta_t {\rm E}\,|\, {\rm E}(\sfC_t) \in [n\eps,(n+1)\eps)\big]\\
		&=\sum_{z} \bbE \big[\Delta_t {\rm E}\,|\, {\rm Ext}(\sfC_t) =z \big] 
		\bbP \big[{\rm Ext}(\sfC_t) =z\,|\, {\rm E}(\sfC_t) \in [n\eps,(n+1)\eps)\big], 
		\end{aligned}
	\end{align}
	where the sum is taken over all points $z \in \bbL_t$ with $\langle z,e_\theta \rangle \in [n\eps,(n+1)\eps)$.
	The summands will be controlled in different ways, depending on the value of $z$. In the following, $z$ is always assumed to satisfy $\langle z,e_\theta \rangle \in [n\eps,(n+1)\eps)$.
	
	Choose $1 \leq r \leq R$ depending on $\eps$ and $\delta$; the choice of $r$ and $R$ will be explained below. 
	The constant $q_0 > 4$ below will depend on $\eps, \delta$ and on $r$ and $R$, but not on the angles $\theta, \alpha, \beta$, nor on $n$ and $N$ above some threshold. 
	We henceforth assume $\eps n > R$, with further lower bounds imposed below. 
	
	We distinguish three scenarios depending on the location of $z = {\rm Ext}(\sfC_t)$.
	
	\begin{enumerate}[label=(\arabic*)]
		\item $z$ is in a pure block, at a distance at least $4$ from any interface; 
		\item $z$ is within distance $R$ from an interface, but not as in the first case;
		\item or $z$ is in the mixed block, at a distance at least $R$ from any interface.
	\end{enumerate}
		
	Before describing how to bound the summands in \eqref{eq:Delta_sum}, we state a separation lemma that will be useful in the following proofs.
	
	Recall from Section~\ref{sec:background_q=4} the definition of a good flower domain.
	
	\begin{lem}\label{lem:well-sep}
		There exists $c > 0$ such that, for any $\rho \geq 1$ the following holds. 
		There exists $q_0 >4$ such that, for any $q \in [4,q_0)$ and any $z$ at distance at least $4\rho$ from the horizontal axis,
		\begin{align}\label{eq:well-sep}
			\bbP\big[\text{the flower domain in $\om_t$ from $\La_{2\rho}(z)$ to $\La_{\rho}(z)$ is good } \,\big|\, {\rm Ext}(\sfC_t) = z \big] \geq c.
		\end{align}
	\end{lem}

	\begin{proof}
		For configurations sampled according to the measure with $q = 4$ the statement holds according to Lemma~\ref{lem:good flower domain}. Adapting this to our setting, we find that there exists a constant $c >0$ such that, for any $\rho\geq1$ and $z$ at a distance at least $4\rho$ from the horizontal axis, 
		any $t \geq 0$, and any configuration $\tilde \omega$ on $\La_{4\rho}(z)^c$,
		\begin{align}\label{eq:good_scale_q4}
			\phi_{\bbL_t,4}^0 \big[\calF\text{ is good} \,\big|\, \omega = \tilde\omega \text{ on $\La_{4\rho}(z)^c$ and } {\rm Ext}(\sfC_t) = z \big]\geq 2c,
		\end{align}
		where we write $\calF$ for the flower domain revealed from $\La_{2\rho}(z)$ to $\La_{\rho}(z)$ and where the inequality holds as long as the conditioning is non-degenerate.
		
		Now fix $\rho \geq1$. 
		The event in \eqref{eq:good_scale_q4} depends only on the configuration in $ \La_{4\rho}(z)$ with
		the measure in this region being of the form $\phi_{\La_{4\rho}(z),4}^\xi$ with boundary conditions $\xi$ induced by $\tilde\omega$ 
		and a conditioning on an event that depends on $\tilde\omega$. 
		
		The measures $\phi_{\La_{4\rho}(z),  q  }^\xi$ are all continuous in $q$. 
		As such, there exists $q_0 >4$ such that, for any $q \in (4,q_0]$, any boundary conditions $\xi$ and any non-degenerate event $H$,
		\begin{align}
			d_{\rm TV}\big(\phi_{\La_{4\rho}(z),q}^\xi[\, \cdot \,|\, H],\phi_{\La_{4\rho}(z),4}^\xi[\,\cdot \,|\, H]\big) \leq c
		\end{align}
		The above refers to the distance in total variation between the conditioned measures. 
		Combining the above with \eqref{eq:good_scale_q4}, we find
		\begin{align}\label{eq:good_scale_q42}
			\phi_{\bbL_t,  q  }^0 \big[\calF\text{ is good} \,\big|\, \omega = \tilde\omega \text{ on $\La_{4\rho}(z)^c$ and } {\rm Ext}(\sfC_t) = z \big]\geq c, \quad
		\end{align}
		for any configuration $\tilde \omega$ on  $\La_{4\rho}(z)^c$ such that the conditioning is non-degenerate.
		Finally, as $\omega_t$ follows the law $\phi_{\bbL_t,  q  }^0$, \eqref{eq:good_scale_q42} directly implies the desired bound. 
	\end{proof}
		
	We now return to the proof of Proposition~\ref{prop:conditional_drift} and deal with the previously defined scenarios individually

	\paragraph{Case (1):} 
	For $z$ in a pure block, at a distance at least $4$ from an interface, we have
	\begin{align}\label{eq:case frozen block}
		\bbP \big[\Delta_t {\rm E} \geq 0 \,|\, {\rm Ext}(\sfC_t) =z \big]  = 1.
	\end{align}
	Indeed, the cell of $z$ is not affected by the transformation $(\bfS_{t+1} \circ \bfS_t)$ and therefore intersects $\sfC_{t+2}$. 
	The inequality stems from the fact that some other point may overtake the extremum and thus increase the extremal coordinate.
	
	\paragraph{Case (2):} 
	We argue that the probability for ${\rm Ext}(\sfC_t)$ to be in case $2$ under the conditioning ${\rm E}(\sfC_t) \in [n\eps,(n+1)\eps)$ is small. 
	Combining this with the deterministic lower bound $\Delta_t {\rm E} \geq -4$ shows that the contribution of points $z$ in the second case to \eqref{eq:Delta_sum} is not substantially negative. 
	
	\begin{lem}\label{lem:extremum_close_to_interface}
		There exist constants $C> 1$ and $q_0 > 4$ such that, for any $q \in (4,q_0]$, if $n,N$ are sufficiently large, 
		\begin{align}\label{eq:extremum_close_to_interface}
			\begin{aligned}
			\bbP\big[{\rm Ext}(\sfC_t) \text{ is at a distance at most $R$ from an interface,} & \\ \text{ but at least $CR$ from the horizontal axis} &\,\big|\, {\rm E}(\sfC_t) \in [n\eps,(n+1)\eps) \big] < \delta.
			\end{aligned}
		\end{align}
	\end{lem}
	
	\begin{proof}[Proof of Lemma~\ref{lem:extremum_close_to_interface}]
		Fix $C \geq 1$; we will see below how to choose it. We henceforth assume $N > 2CR$. 
		Additionally, we will assume

		Consider a potential realisation $z$ of ${\rm Ext}(\sfC_t)$ with $ \langle z,e_\theta\rangle \in [n\eps,(n+1)\eps)$ and $z$ at a distance at least $4C R$ from the horizontal axis
		and with the minimal distance between $z$ and an interface being smaller than $R$. Call the latter distance $d$.
		
		Consider the following exploration procedure of the unconditioned configuration $\omega_t$.
		Explore the flower domain $\calF$ from $\La_{2CR}(z)$ to $\La_{CR}(z)$ and reveal the configuration on $\calF^c$; we call this the explored region. 
		Conditionally on the exploration, we distinguish three scenarios:
		\begin{enumerate}[label=(\roman*)]
			\item $\calF$ is not good; 
			\item the configuration in the explored region is such that, for any configuration in $\calF$, we have ${\rm Ext}(\sfC_t) \neq z$;
			\item the configuration in the explored region is such that the cluster of the origin is connected to the primal petal of $\calF$ and is contained in $\calH_{< n\eps}^{\theta}$. In this case, depending on the configuration in $\calF$, we may have ${\rm Ext}(\sfC_t) = z$.
		\end{enumerate}
		
		Let $c_{\rm good} > 0$ be the constant given by Lemma~\ref{lem:well-sep}. Then there exists $\tilde q_0 > 4$ (depending on $c_{\rm good}$ and $CR$, but not on $z$ or any other fixed quantities) such that
		\begin{align}
			\bbP\big[\text{Case (i)} \,|\,{\rm Ext}(\sfC_t) = z  \big] \leq 1 - c_{\rm good}
		\end{align}
		for any $q \in (4,\tilde q_0]$.

		In case (ii), we have
		\begin{align}
			\bbP\big[{\rm Ext}(\sfC_t) = z \,|\, \calF \big] = 0,
		\end{align}
		where the conditioning is that $\calF$ is the result of the exploration procedure described above.
		
		We conclude that 
		\begin{align}\label{eq:caseiii_and}
			\bbP\big[\text{Case (iii) and } {\rm Ext}(\sfC_t) = z  \big] \geq c_{\rm good}\cdot	\bbP\big[ {\rm Ext}(\sfC_t) = z  \big].
		\end{align}
		
		Now fix a realisation of $\calF$ as in case (iii). 
		The measure for $\omega_t$ inside $\calF$ is then given by $\phi_{\calF,q}^\xi$, where $\xi$ are the boundary conditions induced by the exploration. 
		For configurations in $\calF$, write ${\rm Ext}(\xi)$ for the extremum of the cluster of the primal petal in the direction $e_\theta$. 
		If $\omega_t$ is such that $\langle {\rm Ext}(\xi),e_\theta \rangle \in [n\eps,(n+1)\eps)$, 
		then ${\rm Ext}(\sfC_t) = {\rm Ext}(\xi)$.
	
		This applies in particular to  ${\rm Ext}(\xi) = z$, but also to other points $z'$ in the strip $\langle z',e_\theta \rangle \in [n\eps,(n+1)\eps)$.
		Consider now $4 < q_0 \leq \tilde q_0$ (depending on $CR$) such that, for any $4 \leq q \leq q_0$ and any $\calF$ and $z'$ as above, 
		\begin{align}
			\tfrac12 \leq \frac{\phi_{\calF,q}^\xi [{\rm Ext}(\xi) = z'] }{\phi_{\calF,4}^\xi [{\rm Ext}(\xi) = z']} \leq 2.
		\end{align}
		The existence of such a value $q_0$ is guaranteed by the continuity in $q$ of the measures $\phi_{\calF,q}^\xi$.

		Standard applications of \eqref{eq:RSW}, \eqref{eq:mixing} and \eqref{eq:quasi_multiplicativity}  show that, for any $z' \in \La_{CR/2}(z)$ with $\langle z',e_\theta \rangle \in [n\eps,(n+1)\eps)$,
		if we write $d'$ for the distance between $z'$ and the interface\footnote{We assume here that there is a unique interface intersecting $\calF$. Due to the assumption that $N \geq 2CR$, there exist at most two such interfaces. The case of two interfaces may be treated in a similar way, but is omitted here for simplicity.}
		\begin{align}
			c_0 \leq \frac{\phi_{\calF,4}^\xi [ {\rm Ext}(\xi) = z']} 
			{\phi_{\bbL_{\rm mix}, 4} \big[A_3^{{\rm hp}(\theta)} (z'; 1,d')\big] \phi_{\bbL_{t}, 4}^0 \big[A_3^{{\rm hp}(\theta)} (z'; d', CR)\big]}
			\leq \tfrac1{c_0},
		\end{align}
		for some universal constant $c_0 > 0$. 
		
		Applying the above to $z' =z$ and using Proposition~\ref{prop:arm_exp}, we find 
		\begin{align}\label{eq:ext_prim_z}
			\phi_{\calF,q}^\xi [{\rm Ext}(\xi) = z] 
			\leq C_0  \, d^{-2}\big(\tfrac{d}{ CR })^{1+c} = C_0  \,   d^{-1+c}{ (CR) }^{-1-c}. 
		\end{align}
		for some universal constant $C_0 >0$ and with the constant $c >0$ given by \eqref{eq:generic 3-arm}.
		
		Conversely, for any fixed $z'$ as above with $d' \geq  CR /4$, we have 
		\begin{align}
			\phi_{\calF,q}^\xi \big[{\rm Ext}(\xi) = z'\big]  \geq  c_1 ( CR) ^{-2}.
		\end{align}
		for some universal constant $c_1 > 0$.
		There exists a constant\footnote{It is here that it is essential that we condition on a rounding of $e_t$ and not its exact value.}  $c_2 = c_2(\eps) > 0$, independent of any other quantity except $\eps$,  
		such that the number of vertices $z' \in \La_{ CR /2}(z)$ with $\langle z',e_\theta \rangle \in [n\eps,(n+1)\eps)$ and  $d' \geq  CR /4$ is at least $c_2CR $. 
		Thus, summing over all such $z'$, we conclude that 
		\begin{align}\label{eq:ext_prim_z2}
			\phi_{\calF,q}^\xi\big[\langle {\rm Ext}(\xi), e_\theta \rangle \in [n\eps,(n+1)\eps)\big]   \geq  c_2c_1  (CR) ^{-1}.
		\end{align}
		
		Combining \eqref{eq:ext_prim_z} and \eqref{eq:ext_prim_z2}, we conclude that
		\begin{align}
			\bbP\big[{\rm Ext}(\sfC_t) = z \,|\, \calF \big] 
			&= \phi_{\calF,q}^\xi \big[ {\rm Ext}(\xi) = z \big] \\
			&\leq c_3\,   d^{-1+c}{ (CR) }^{-c}			\phi_{\calF,q}^\xi \big[ \langle {\rm Ext}(\xi) ,e_\theta \rangle \in [n\eps,(n+1)\eps) \big]\\
			& = c_3\,   d^{-1+c}{(CR)}^{-c} \bbP\big[  {\rm E}(\sfC_t) \in [n\eps,(n+1)\eps) \,|\,  \calF \big],
		\end{align}
		where $c_3\, = \frac{C_0}{c_1c_2} > 0$ is a universal constant. The conditioning is again that $\calF$ is the result of the exploration procedure described above. 
		
		As the above is valid for any explored flower domain $\calF$ in case (iii), we conclude that
		\begin{align}
		\bbP\big[\text{Case (iii) and } {\rm Ext}(\sfC_t) = z  \big]  
		\leq  c_3\,   d^{-1+c}(C R)^{-c}   \bbP\big[  {\rm E}(\sfC_t) \in [n\eps,(n+1)\eps)\big].
		\end{align}
		Finally, combining this with \eqref{eq:caseiii_and}, we have 
		\begin{align}
		\bbP\big[{\rm Ext}(\sfC_t) = z \,\big|\,  {\rm E}(\sfC_t) \in [n\eps,(n+1)\eps) \big]  
		\leq \tfrac{1}{c_{\rm good}}\,  c_3\,   d^{-1+c}(C R)^{-c}  .
		\end{align}
		Summing over $z$ with $d \leq R$, $\langle z,e_\theta\rangle \in [n\eps,(n+1)\eps)$, and which are at a distance at least $4CR$ from the horizontal axis, we find 
		\begin{align}
			\bbP\big[{\rm Ext}(\sfC_t) \text{ is at a distance at most $R$ from an interface,} & \\ \text{ but at least $4CR$ from the horizontal axis} &\,\big|\, {\rm E}(\sfC_t) \in [n\eps,(n+1)\eps) \big] 
			\leq   c_4   C^{-c},
		\end{align}
		where $c_4$ is a universal constant and $c >0$ is given by \eqref{eq:generic 3-arm}. 
		By choosing $C$ sufficiently large, we may render the right-hand side of the above smaller than $\delta$, thus proving~\eqref{eq:extremum_close_to_interface} with $4C$ instead of~$C$.
			\end{proof}
			
	We return now to the analysis of case (2). 
	Let $C \geq1$ and $q_0 > 4$ be the constants given by Lemma~\ref{lem:extremum_close_to_interface}.
	Assume henceforth that $4  < q \leq q_0$, and that $N$ is such that $\delta N \sin \alpha > CR+4$ and $\delta N \sin \beta > CR+4$. 
	For $ t \leq (1-\delta)N$ the tracks $t_0,\dots, t_{\delta N}$ are part of the frozen block of angle $\alpha$; 
	for  $ t \geq (1+\delta)N$, they are part of the frozen block of angle $\beta$. 
	In both cases, due to our assumption on $N$, these blocks have a height at least $CR$. 	
	Thus, case (2) implies that the extremum is within distance $R$ of an interface, but also at a distance at least $CR$ from the horizontal axis. 
	Combining the deterministic bound $\Delta_t {\rm E} \geq -4$ with \eqref{eq:extremum_close_to_interface}, we find
	\begin{align}\label{eq:case close to interface}
		\sum_{z \text{ in case (2)}} \bbE \big[\Delta_t {\rm E}\,|\, {\rm Ext}(\sfC_t) =z \big] 
		\bbP \big[{\rm Ext}(\sfC_t) =z\,|\, {\rm E}(\sfC_t) \in [n\eps,(n+1)\eps)\big] \geq -4\delta,
	\end{align}
	where the sum is taken over all possible realisations $z$ of ${\rm Ext}(\sfC_t)$ included in case (2).
	
	\paragraph{Case (3):} The remaining points $z$ are in the mixed block, at a distance at least $R$ from any interface. In this scenario, we want to relate the increment $\Delta_t {\rm E}$ to the increment $\Delta^{\rm IIC} {\rm E}$ of an IIC (see Section~\ref{sec: IIC}).
	
	To that end, we first prove that the local environment around an extremum is indistinguishable from that of an IIC extremum (up to an arbitrarily small error).
	
	\begin{lem}\label{lem:IIC_mixing}
		For any $r \geq1$, there exists a choice of $R \geq r$ and $q_0 >4$ such that, for any $q \in [4,q_0)$, any $t \geq 0$ even,
		and any $z \in \bbL_t$ in the mixed block, at a distance at least $R$ from an interface, 
		\begin{align}
			d_{\rm TV}\big(\bbP\big[\cdot \,|\, {\rm Ext}(\sfC_t) = z \big], \phi_{\bbL_{\rm mix},  4  }^{\rm IIC}\big) \leq \delta,
		\end{align}
		where the two measures refer to the configuration in $\La_r(z)$ and the latter is translated by $z$. 
	\end{lem}
	
	Note here that $R$ and $q_0$ depend on $r$ and $\delta$, but not on $z$, $t$ or any other quantity previously fixed. 
	
	\begin{proof}
		Fix $r \geq 1$ and let $R \geq r$ be a constant to be fixed below. Fix a point $z$ as in the statement. 
		
		First, observe that, in $\bbL_t$, the lattice in $\La_R(z)$ is identical to $\bbL_{\rm mix}$ --- including its partition into cells. 
		Thus, Proposition~\ref{prop:IIC mixing at q=4} ensures that, by choosing $R = R(r,\delta) \geq r$ sufficiently large,
		\begin{align}
			d_{\rm TV}\big(\phi_{\bbL_t,4}^0\left[\, \cdot \,|\, \om = \om_0 \text{ on } \La_R(z)^c, {\rm Ext}(\sfC_t) = z\right]
			,\phi^{\rm IIC}_{\bbL_{\rm mix},4}\big) \leq \tfrac\delta2
		\end{align}
		for any configuration $\omega_0$ on $\La_R(z)^c$ for which the conditioning is not degenerate, where the two measures refer to the configuration in $\La_r(z)$ and the latter is translated by $z$.
		 
		 Now, by continuity of the measures $\phi_{\La_R(z),  q}^\xi$, there exists $q_0 = q_0(R,\delta)>4$ such that, for any $4 <q \leq q_0$, 
		\begin{align}
			d_{\rm TV}\big(\phi_{\bbL_t,4}^0[\,\cdot\,|\, \om = \om_0 \text{ on } \La_R(z)^c, {\rm Ext}(\sfC_t) = z], \phi_{\bbL_t,q}^0[\,\cdot\,|\, \om = \om_0 \text{ on } \La_R(z)^c, {\rm Ext}(\sfC_t) = z]\big) \leq \tfrac\delta2
		\end{align}
		for any $\omega_0$ as above, 
		with both measures referring to the configuration in the full box $\La_R(z)$. 
		 
		Combining the two displays above, and keeping in mind that the law of $\omega_t$ is $\phi_{\bbL_t,  q }^0$, we obtain the desired conclusion. 
	\end{proof}
	
	We now aim to control the terms $\bbE[\Delta_t{\rm E} \,|\, {\rm Ext}(\sfC_t) = z]$ for $z$ deep within the mixed block. 
	Sample $\om^{\rm IIC}$ according to $\phi^{\rm IIC}_{\bbL_{\rm mix},  4  }$ and translate it by $z$. 
	Include this sample under the measure $\bbP$ so as to maximise the probability under $\bbP[\,\cdot\,|\, {\rm Ext}(\sfC_t) = z]$
	that $\omega_t$ and  $\om^{\rm IIC}$ are identical in $\La_r(z)$. 
	
	Write $\sfC^{\rm IIC}$ for the cluster of $z$ in $\om^{\rm IIC}$ and $\tilde \sfC^{\rm IIC}$ for the corresponding cluster in $(\bfS_{t+1}\circ\bfS_{t})(\om^{\rm IIC})$. Notice that the effect of $(\bfS_{t+1} \circ \bfS_{t})$ on the local environment of ${\rm Ext}(\sfC^{\rm IIC})$ is identical to that of $(\bfS_{\rm odd} \circ \bfS_{\rm even})$. In particular, we have that
	\begin{align}
		{\rm E}(\tilde\sfC^{\rm IIC})-{\rm E}(\sfC^{\rm IIC})= \Delta^{\rm IIC}{\rm E}.
	\end{align}
		
	There are multiple reasons why the increment $\Delta_t{\rm E}$ might differ from $\Delta^{\rm IIC}{\rm E}$. 
	We say a \emph{coupling error} occurs if the configurations $\om_t$ and $\om^{\rm IIC}$ are not identical in $\La_r(z)$. 
	An \emph{increment error} occurs if the configurations $(\bfS_{t+1}\circ\bfS_t)(\om^{\rm IIC})$ and $\om_{t+2} = (\bfS_{t+1} \circ \bfS_t)(\om_t)$ are not identical on $\La_{r/2}(z)$. 
	Finally, an \emph{IIC error} occurs if the extremum of $ \tilde\sfC^{\rm IIC}$ is not contained in $\La_{r/2}(z)$.
	
	If none of these errors occur, then $\om_{t+2}$ is identical to $(\bfS_{t+1}\circ\bfS_t)(\om^{\rm IIC})$ on $\La_{r/2}(z)$ and said box contains ${\rm Ext}( \tilde \sfC^{\rm IIC})$. In particular, ${\rm Ext}(\tilde \sfC^{\rm IIC}) \in \sfC_{t+2}$ and therefore
	\begin{align}
		{\rm E}(\sfC_{t+2}) \geq {\rm E}( \tilde\sfC^{\rm IIC}).
	\end{align}
	The inequality comes from the case where $\sfC_{t+2}$ has an extremum outside of $\La_{r/2}(z)$. 
	While we expect this to be unlikely under $\bbP[\, \cdot \,|\, {\rm Ext}(\sfC_t) = z ]$, we will not endeavour to bound its probability. 
	
	By collecting the error terms, we obtain
	\begin{align}
		\bbE\big[ \Delta_t {\rm E} \,|\, {\rm Ext}(\sfC_t) = z \big] &\geq \bbE \big[ \Delta^{\rm IIC}{\rm E} - 4 \ind_{\text{error}} \,|\, {\rm Ext}(\sfC_t) = z \big] \\
		&\geq  -4\bbP\big[ \text{error} \,|\, {\rm Ext}(\sfC_t) = z\big],\label{eq:errorsDelta}
	\end{align}
	where the first inequality is due to the deterministic bound on increments and the second one to the fact that that $\bbE[\Delta^{\rm IIC}{\rm E}] =  0$ --- see \eqref{eq:IIC drift is null}.
	
	We will now bound the probabilities of each type of error occurring.
	We start with the IIC and increment errors, as these are controlled by taking $r$ large enough. 
	Indeed, since the track exchanges are \emph{local} transformations, the probability under $\bbP[\,\cdot \,|\, {\rm Ext}(\sfC_t) = z ]$  that an increment error occurs is bounded by $Ce^{-cr}$ for universal constants $C$ and $c$ --- see Proposition~\ref{prop:track-exchange operator}.
	For an IIC error to occur, $\sfC^{\rm IIC}$ should contain a point $z' \notin \La_{r/2}(z)$ with 
	$$ \langle z', e_\theta \rangle \geq  \langle z, e_\theta \rangle -2.$$
	By a union bound on the possible values of $z'$ and using Proposition~\ref{prop:arm_exp}, the probability of an IIC error occurring may be bounded by $Cr^{-1}$ for some universal constant $C$.
	This computation is identical to the corresponding one in \cite{DCKozKraManOulRotationalInvarianceCritPlanar}. 
	Thus, by taking $r \geq 1$ sufficiently large we may ensure that 
	\begin{align}
	\bbP\big[ \text{increment or IIC error} \,|\, {\rm Ext}(\sfC_t) = z\big] \leq \delta. 
	\end{align}
	Finally, with $r$ fixed, Lemma~\ref{lem:IIC_mixing} states that one may choose $R \geq r$ large enough and $q_0 > 4$ such that, 
	\begin{align}
		\bbP\big[ \text{coupling error} \,|\, {\rm Ext}(\sfC_t) = z\big] \leq \delta. 
	\end{align}
	
	Inserting the last two bounds in \eqref{eq:errorsDelta}, we find that, for $R\geq r\geq1$ chosen as above, and for $q \in (4,q_0]$, with $q_0$ as above, 
	\begin{align}\label{eq:case mixed block}
		\bbE\big[ \Delta_t {\rm E} \,|\, {\rm Ext}(\sfC_t) = z \big] \geq -8\delta.
	\end{align}
	
	\paragraph{Conclusion of proof of \eqref{eq:conditional_drift}:}
	Take $R\geq r\geq1$ as dictated by case (3), so that \eqref{eq:case mixed block} holds.
	Consider now $q_0 >4$ sufficiently close to $4$ so that  \eqref{eq:case frozen block}, \eqref{eq:case close to interface}, and \eqref{eq:case mixed block} hold for all $4 < q<q_0$, with the values of $r,R$ chosen above. 
	Then, summing these bounds we find
		\begin{align}
			\bbE \big[\Delta_t {\rm E}\,|\, {\rm E}(\sfC_t) \in [n\eps,(n+1)\eps)\big] 
			\geq  - 12 \delta,
		\end{align}
		as required. \hfill$\square$ 
	
	\subsection{Deducing universality: proof of Theorem~\ref{thm:universality}}\label{sec:conclusion_universality}
	
	With Corollary~\ref{cor: extremum barely moves} at hand, we are almost ready to prove our main result.
	We will prove Proposition~\ref{prop:universality} and then see how it implies Theorem~\ref{thm:universality}. 

	\begin{proof}[Proof of Proposition~\ref{prop:universality}]
		Fix $\eps > 0$ and $\alpha, \beta \in (\eps,\pi-\eps)$ and $\theta \neq \frac{3\pi}2$.
		All constants below will be independent of $\alpha$, $\beta$ and $\theta$, but may depend on $\eps$. 
		
		Due to \eqref{eq:exponential_decay_L(alpha)}, we may fix $C$ such that, for any $q > 4$ and any $n\geq 1$ large enough\footnote{The lower bound on $n$ may depend on $q$, as we will take $n$ to infinity first.}, 
		\begin{align}
			\phi^0_{\bbL_+(\alpha),q}\big[0 \lr  \calH_{\geq n}^{\theta} \big] \leq 2 \phi^0_{\bbL_+(\alpha),q}\big[0 \lr  \calH_{\geq n}^{\theta} \text{ below $t_{Cn}$}\big]
		\end{align}
		and the same for $\bbL_+(\beta)$. 
		
		Set $N = Cn$ and define the lattices $\bbL_t$ accordingly. 
		Then, by the above and Corollary~\ref{cor:top_does_not_matter}, 
		\begin{align}
			\begin{aligned}
			\bbP\big[{\rm E}(\sfC_{0}) \geq n\big] &\geq \tfrac12 \phi^0_{\bbL_+(\alpha),q}\big[0 \lr  \calH_{\geq n}^{\theta}\big] 
			\quad \text{ and} 		\\
			\bbP\big[{\rm E}(\sfC_{2N}) \geq n(1-\eps)\big] &\leq 2 \phi^0_{\bbL_+(\beta),q}\big[0 \lr  \calH_{\geq n(1-\eps)}^{\theta}\big].
			\label{eq:full_alpha_to_mix}
			\end{aligned}
		\end{align}
		
		Let $q_0 >4$ be given by Corollary~\ref{cor: extremum barely moves} for $\eta = \eps$ and $C$ fixed as above. 
		Fix $q \in (4,q_0]$.
		By \eqref{eq:extremum barely moves} and \eqref{eq:full_alpha_to_mix}, we have
		\begin{align}\label{eq:hp universality}
			\phi^0_{\bbL_+(\beta),q}\big[0 \lr  \calH_{\geq n(1-\eps)}^{\theta}\big]\geq \tfrac\eps 4\phi^0_{\bbL_+(\alpha),q}\big[0 \lr  \calH_{\geq n}^{\theta}\big]
		\end{align}
		for all $n$ sufficiently large. 
		By taking the logarithm, dividing by $-n(1-\eps)$, and taking the limit $n \to \infty$ in \eqref{eq:hp universality}, we obtain
		\begin{align}
			(1-\eps)\cdot  \zeta_{\beta,q}^{\rm hp}(\theta)^{-1}  \leq \zeta_{\alpha,q}^{\rm hp}(\theta)^{-1}.
		\end{align}
		By inverting the roles of $\alpha$ and $\beta$ in the coupling and repeating the same argument, we obtain the opposite bound.
	\end{proof}

	\begin{proof}[Proof of Theorem~\ref{thm:universality}]
		We start with proving the statement about $\zeta$. 
		Fix $\eps > 0$ and let $q_0 >4$ be given by Proposition~\ref{prop:universality}. Henceforth consider $q \in (4,q_0]$. 
		Let $\theta \in [0,2\pi)$ and choose $\tilde\theta \in \{\theta,\theta+\pi\}$ (mod $2\pi$) to be upper-half-plane aiming for $\bbL(\alpha)$.
		Then, by Proposition~\ref{prop:hp_cor_len} and Proposition~\ref{prop:universality} applied to~$\tilde\theta$, 
		\begin{align}
			\zeta_{\alpha,q}(\theta) 
			= \zeta_{\alpha,q}(\tilde\theta)   
			= \zeta_{\alpha,q}^{\rm hp}(\tilde\theta)   
			\leq (1+\eps)\, \zeta_{\frac{\pi}2,q}^{\rm hp}(\tilde\theta)  
			\leq (1+\eps)\,\zeta_{\frac{\pi}2,q}(\tilde\theta)  
			=(1+\eps)\, \zeta_{\frac{\pi}2,q}(\theta).
		\end{align}
		Applying the same reasoning with the roles of $\alpha$ and $\pi/2$ exchanged we obtain the opposite bound. Note that the value of $\tilde\theta$ may a priori change for this second case.
		
		Finally, the statement about $\xi$ is readily deduced from that about $\zeta$ using \eqref{eq:convex-duality_iso}.
	\end{proof}

	\paragraph{Acknowledgements.} The authors thank S\'ebastien Ott for useful discussions. This work was supported by the Swiss National Science Foundation and the swissuniversities ``cotutelle de th\`{e}se'' grant. 

	\bibliographystyle{alpha}
	\bibliography{biblio.bib}
	
\end{document}